\documentclass[reqno]{amsart}
\usepackage{hyperref}

\AtBeginDocument{{\noindent\small
\emph{}}
\vspace{8mm}}
\begin{document}
\title[\hfil fractional Schr\"{o}dinger-Poisson system]
{Concentration of bound states for fractional Schr\"{o}dinger-Poisson system via penalization methods}

\author[K. M. Teng ]
{Kaimin Teng}  
\address{Kaimin Teng (Corresponding Author)\newline
Department of Mathematics, Taiyuan
University of Technology, Taiyuan, Shanxi 030024, P. R. China}
\email{tengkaimin2013@163.com}

\subjclass[2010]{35B38, 35R11}
\keywords{Fractional Schr\"{o}dinger-Poisson system; concentration; bound state; penalization method.}

\begin{abstract}
In this paper, we study the following fractional Schr\"{o}dinger-Poisson system
\begin{equation*}
\left\{
  \begin{array}{ll}
    \varepsilon^{2s}(-\Delta)^su+V(x)u+\phi u=g(u) & \hbox{in $\mathbb{R}^3$,} \\
     \varepsilon^{2t}(-\Delta)^t\phi=u^2,\,\, u>0& \hbox{in $\mathbb{R}^3$,}
  \end{array}
\right.
\end{equation*}
where $s,t\in(0,1)$, $\varepsilon>0$ is a small parameter. Under some local assumptions on $V(x)$ and suitable assumptions on the nonlinearity $g$, we construct a family of positive solutions $u_{\varepsilon}\in H_{\varepsilon}$ which concentrates around the global minima of $V(x)$ as $\varepsilon\rightarrow0$.
\end{abstract}

\maketitle
\numberwithin{equation}{section}
\newtheorem{theorem}{Theorem}[section]
\newtheorem{lemma}[theorem]{Lemma}
\newtheorem{definition}[theorem]{Definition}
\newtheorem{remark}[theorem]{Remark}
\newtheorem{proposition}[theorem]{Proposition}
\newtheorem{corollary}[theorem]{Corollary}
\allowdisplaybreaks

\section{Introduction}
In this paper, we study the following fractional Schr\"{o}dinger-Poisson system
\begin{equation}\label{main}
\left\{
  \begin{array}{ll}
    \varepsilon^{2s}(-\Delta)^su+V(x)u+\phi u=g(u) & \hbox{in $\mathbb{R}^3$,} \\
     \varepsilon^{2t}(-\Delta)^t\phi=u^2,\,\, u>0& \hbox{in $\mathbb{R}^3$,}
  \end{array}
\right.
\end{equation}
where $s,t\in(0,1)$, $\varepsilon>0$ is a small parameter. The potential $V:\mathbb{R}^3\rightarrow\mathbb{R}$ is a bounded continuous function satisfying\\
$(V_0)$ $\inf\limits_{x\in\mathbb{R}^3}V(x)=V_0>0$;\\
$(V_1)$ There is a bounded domain $\Lambda\subset\mathbb{R}^3$ such that
\begin{equation*}
V_0<\min_{\partial\Lambda}V(x),\quad \mathcal{M}=\{x\in\Lambda\,\,|\,\, V(x)=V_0\}\neq{\O}.
\end{equation*}
Without of loss of generality, we may assume that $0\in\mathcal{M}$. The nonlinearity $g:\mathbb{R}\rightarrow\mathbb{R}$ is of $C^1$-class function satisfying\\
$(g_0)$ $\lim\limits_{\tau\rightarrow0^{+}}\frac{g(\tau)}{\tau}=0$;\\
$(g_1)$ $\lim\limits_{\tau\rightarrow+\infty}\frac{g'(\tau)}{\tau^{2_s^{\ast}-2}}=0$, where $2_s^{\ast}=\frac{6}{3-2s}$;\\
$(g_2)$ there exists $\lambda>0$ such that $g(\tau)\geq\lambda \tau^{q-1}$ for some $\frac{4s+2t}{s+t}<q<2_s^{\ast}$ and all $\tau\geq0$;\\
$(g_3)$ $\frac{g(\tau)}{\tau^{q-1}}$ is non-decreasing in $\tau\in (0,+\infty)$.\\
Since we are looking for positive solutions, we may assume that $g(\tau)=0$ for $\tau<0$.
The non-local operator $(-\Delta)^s$ ($s\in(0,1)$), which is called fractional Laplacian operator, can be defined by
\begin{equation*}
(-\Delta)^su(x)=C_s\,P.V.\int_{\mathbb{R}^3}\frac{u(x)-u(y)}{|x-y|^{3+2s}}\,{\rm d}y=C_s\lim_{\varepsilon\rightarrow0}\int_{\mathbb{R}^3\backslash B_{\varepsilon}(x)}\frac{u(x)-u(y)}{|x-y|^{3+2s}}\,{\rm d}y
\end{equation*}
for $u\in\mathcal{S}(\mathbb{R}^3)$, where $\mathcal{S}(\mathbb{R}^3)$ is the Schwartz space of rapidly decaying $C^{\infty}$ function, $B_{\varepsilon}(x)$ denote an open ball of radius $r$ centered at $x$ and the normalization constant $C_s=\Big(\int_{\mathbb{R}^3}\frac{1-\cos(\zeta_1)}{|\zeta|^{3+2s}}\,{\rm d}\zeta\Big)^{-1}$. For $u\in\mathcal{S}(\mathbb{R}^3)$, the fractional Laplace operator $(-\Delta)^s$ can be expressed as an inverse Fourier transform
\begin{equation*}
(-\Delta)^su=\mathcal{F}^{-1}\Big((2\pi|\xi|)^{2s}\mathcal{F}u(\xi)\Big),
\end{equation*}
where $\mathcal{F}$ and $\mathcal{F}^{-1}$ denote the Fourier transform and inverse transform, respectively. If $u$ is sufficiently smooth, it is known that (see \cite{NPV}) it is equivalent to
\begin{equation*}
(-\Delta)^su(x)=-\frac{C_s}{2}\int_{\mathbb{R}^3}\frac{u(x+y)+u(x-y)-2u(x)}{|x-y|^{3+2s}}\,{\rm d}y.
\end{equation*}
By a classical solution of \eqref{main}, we mean two continuous functions $u$ and $\phi$ that $(-\Delta)^su$ and $(-\Delta)^t\phi$ are well defined for all $x\in\mathbb{R}^3$ and satisfy \eqref{main} in a pointwise sense.

In recent years, much attention has been given to nonlocal problems driven by the fractional Laplace operator. This operator naturally arises in many physical phenomena, such as: fractional quantum mechanics \cite{La,La1}, anomalous diffusion \cite{MK}, financial \cite{CT}, obstacle problems \cite{S}, conformal geometry and minimal surfaces \cite{CM}. It also provides a simple model to describe certain jump L\'{e}vy processes in probability theory \cite{BV}. One powerful approach is to use the harmonic extension method developed by Caffarelli and Silvestre \cite{CS}, and this extension method can transform a given nonlocal equation into a degenerate elliptic problem in the half-space with a nonlinear Neumann boundary condition, we refer to interesting readers to see the related works \cite{AM,BCPS,BV,CW,DDW,FLS,Teng-2} and so on. Another approach is that directly investigating the problems in the space $H^s(\mathbb{R}^3)$, the related works can be referred to see \cite{BV,CW,DPDV,Sec,SZ,Teng-1,Teng,Teng1,Teng2,Teng3} and so on. If one chooses the second line, since the fractional Laplacian $(-\Delta)^s$ is a nonlocal operator, more accurate estimates are needed usually.

Formally, system \eqref{main} is regarded as the associated fractional version of the following classical Schr\"{o}dinger-Poisson system
\begin{equation}\label{main-1}
\left\{
  \begin{array}{ll}
    -\varepsilon^2\Delta u+V(x)u+\phi u=g(x,u) & \hbox{in $\mathbb{R}^3$,} \\
    -\varepsilon^2\Delta \phi=u^2& \hbox{in $\mathbb{R}^3$.}
  \end{array}
\right.
\end{equation}
It is well known that system \eqref{main-1} has a strong physical meaning because it appears in semiconductor theory \cite{MRS}. In particular, systems like (1.2) have been introduced in \cite{BF} as a model to describe solitary waves. In \eqref{main-1}, the first equation is a nonlinear stationary equation (where the nonlinear term simulates the interaction between many particles) that is coupled with a Poisson equation, to be satisfied by $\phi$, meaning
that the potential is determined by the charge of the wave function. For this reason, \eqref{main-1} is referred to as a nonlinear Schr\"{o}dinger-Poisson system.

In recent years, there has been increasing attention to systems like \eqref{main-1} when $0<\varepsilon\leq1$ on the existence of positive solutions, ground state solutions, multiple solutions and semiclassical states; see for examples \cite{AR,BF,Ruiz,Ruiz1,WTXZ} and the references therein. Regarding the concentration phenomenon of solutions for Schr\"{o}dinger-Poisson systems like \eqref{main-1}, there has been the object of interest for many authors. Ruiz and Vaira \cite{RV} proved the existence of multi-bump solutions of system
\begin{equation}\label{main-1-1}
\left\{
  \begin{array}{ll}
    -\varepsilon^2\Delta u+V(x)u+\phi u=u^p & \hbox{in $\mathbb{R}^3$,} \\
    -\Delta \phi=u^2& \hbox{in $\mathbb{R}^3$}
  \end{array}
\right.
\end{equation}
for $p\in(1,5)$ and these bumps concentrate around a local minimum of the potential $V$, through using the singular perturbed methods based on a Lyapunov-Schmidt reduction. We refer the interesting readers to see \cite{IV,IV1,S} and the references therein. When $V(x)$ satisfies
\begin{equation}\label{equ1-1}
0<V_0=\inf_{\mathbb{R}^N}V(x)<V_{\infty}:=\liminf_{|x|\rightarrow\infty}V\in(0,+\infty]
\end{equation}
and $g(x,u)=f(u)\in C^1(\mathbb{R}^3)$ verifying
\begin{equation}\label{equ1-2}
\left\{
  \begin{array}{ll}
     & \hbox{$(i)$ $f(t)=o(t^3)$, $f'(s)s^2-3f(s)s\geq Cs^{\sigma}$, $C>0$,} \\
     & \hbox{$(ii)$ $\frac{f(t)}{t^3}$ is increasing on $(0,+\infty)$,}\\
     & \hbox{$(iii)$ $0<\mu F(t)=\mu\int_0^tf(s)\,{\rm d}s\leq f(t)t$ for $\mu>4$,}
  \end{array}
\right.
\end{equation}
He \cite{He} considered the existence and concentration behavior of ground state solutions for a class of Schr\"{o}dinger-Poisson system \eqref{main-1} and proved its solutions concentrating around the global minimum of $V$ as $\varepsilon\rightarrow0$. Wang et al. \cite{WTXZ} studied the following system
\begin{equation}\label{main-2}
\left\{
  \begin{array}{ll}
    -\varepsilon^2\Delta u+V(x)u+\lambda\phi u=b(x)f(u) & \hbox{in $\mathbb{R}^3$,} \\
    -\varepsilon^2\Delta \phi=u^2& \hbox{in $\mathbb{R}^3$,}
  \end{array}
\right.
\end{equation}
where $\lambda\neq0$ is a real parameter, $V(x)$ and $b(x)$ satisfy some global assumptions, $f\in C(\mathbb{R}^3)$ is such that
\begin{equation}\label{equ1-3}
\left\{
  \begin{array}{ll}
     & \hbox{$(i)$ $f(t)=o(t^3)$,} \\
     & \hbox{$(ii)$ $\frac{f(t)}{t^3}$ is increasing on $(0,+\infty)$,}\\
     & \hbox{$(iii)$ $|f(t)|\leq c(1+|t|^{p-1})$ with $p\in(4,6)$, $\lim\limits_{|t|\rightarrow\infty}\frac{F(t)}{t^4}=+\infty$,}
  \end{array}
\right.
\end{equation}
the authors proved that problem \eqref{main-2} exists the least energy solution $u_{\varepsilon}\in H^1(\mathbb{R}^3)$ for $\varepsilon>0$ sufficiently small, and $u_{\varepsilon}$ converges to the least energy solution of the associated limit problem and concentrates to some sets.

In the very recent years, there are much attention to be paid on a similar system like \eqref{main}. For example, when $\varepsilon=1$ in \eqref{main}, in \cite{Teng}, we established the existence of positive ground state solution for the system \eqref{main} with $g(u)=\mu|u|^{p-1}u+|u|^{2_s^{\ast}-2}u$ for some $p\in(1,2_s^{\ast}-1)$ by using the Nehari-Pohozaev manifold combing monotone trick with global compactness Lemma. Using the similar methods, in \cite{Teng1}, positive ground state solutions for the system \eqref{main} with $g(u)=|u|^{p-1}u$ with $p\in(2,2_s^{\ast}-1)$, were established when $s=t$. In \cite{ZDS}, the authors studied the existence of radial solutions for the system \eqref{main} with the nonlinearity $g(u)$ satisfying the subcritical or critical assumptions of Berestycki-Lions type. Regarding the semiclassical state of problem \eqref{main}, there are some results on the existence and multiplicity of solutions. Such as, in \cite{MS}, the authors studied the semiclassical state of the following system
\begin{equation*}
\left\{
  \begin{array}{ll}
    \varepsilon^{2s}(-\Delta)^su+V(x)u+\phi u=f(u) & \hbox{in $\mathbb{R}^N$,} \\
    \varepsilon^{\theta}(-\Delta)^{\frac{\alpha}{2}}\phi=\gamma_{\alpha}u^2& \hbox{in $\mathbb{R}^N$,}
  \end{array}
\right.
\end{equation*}
where $s\in(0,1)$, $\alpha\in(0,N)$, $\theta\in(0,\alpha)$, $N\in(2s,2s+\alpha)$, $\gamma_{\alpha}$ is a positive constant, $V(x)$ satisfies \eqref{equ1-1} and $f(u)$ satisfies
the assumptions like \eqref{equ1-2}. By using the Ljusternick-Schnirelmann theory of critical point theory, the authors obtained the multiplicity of positive solutions which concentrate on the minima of $V(x)$ as $\varepsilon\rightarrow0$. In \cite{LZ}, by using the methods mentioned before, Liu and Zhang proved the existence and concentration of positive ground state solution for problem \eqref{main}. In \cite{Teng2}, we studied the system \eqref{main} with competing potential, i.e., $g(u)=K(x)f(u)+Q(x)|u|^{2_s^{\ast}-2}u$, where $f$ is a function of $C^1$ class, superlinear and subcritical nonlinearity, $V(x)$, $K(x)$ and $Q(x)$ are positive continuous functions. Under some suitable assumptions on $V$, $K$ and $Q$, we prove that there is a family of positive ground state solutions which concentrate on the set of minimal points of $V(x)$ and the sets of maximal points of $K(x)$ and $Q(x)$.

In the above mentioned works, the assumptions made on potential $V(x)$ are all global, but for the local assumption like $(V_1)$, there are few works to deal with the fractional Schr\"{o}dinger-Poisson system \eqref{main}, even for the Schr\"{o}dinger-Poisson system \eqref{main-1}. It is well known that the penalization methods developed by del Pino and Felmer \cite{DF} is a powerful trick to solve this class of problems, but it requires the arguments of Nehari manifold. Recently this powerful tools have been applied to fractional Schr\"{o}dinger equations, see \cite{AM, Ambrosi, HZ}. When using the Nehari manifold for the system \eqref{main-1}, the nonlinearity $g(x,u)$ has to be suplinear-4 growth, i.e., $\lim\limits_{t\rightarrow+\infty}\frac{G(x,t)}{t^4}=+\infty$. The purpose of this paper is to extand the threshold of superliear-4 growth. Another penalization which was developed by Byeon and Jeanjean \cite{BJ} is another effective methods, but this method is not available for the nonlinear problems involving fractional Laplacian since the fractional operator $(-\Delta)^s$ is nonlocal, this makes the function $u$ with $u=0$ on $\mathbb{R}^3\backslash B_R(0)$, satisfies the equation $(-\Delta)^s u=f(u)$ in $B_R(0)$, it will not hold on $\mathbb{R}^3\backslash B_R(0)$ if $f(0)=0$. But for the local Laplace operator $-\Delta$, it possesses this properties which $u$ satisfies the equation $-\Delta u=f(u)$ with $f(0)=0$ in the whole $\mathbb{R}^3$. This property is vital to use the penalization method of Byeon and Jeanjean \cite{BJ}. The penalization used by Byeon and Jeanjean \cite{BJ} is defined by
\begin{equation*}
\chi_{\varepsilon}(x)=\left\{
                        \begin{array}{ll}
                          0 & \hbox{$x\in\Lambda/\varepsilon$,} \\
                          1/\varepsilon & \hbox{$x\not\in\Lambda/\varepsilon$,}
                        \end{array}
                      \right.\quad Q_{\varepsilon}(u)=\Big(\int_{\mathbb{R}^3}\chi_{\varepsilon}u^2\,{\rm d}x-1\Big)_{+}^2.
\end{equation*}
To obtain the $L^{\infty}$-estimates and uniformly decay estimate at infinity, this penalization can not applicable directly because there is no local estimates like Theorem 8.17 in \cite{GT}. For tackling these difficulties, we combine the two penalizations mentioned above which has been introduced in Byeon and Wang \cite{BW}, but a change of second penalization is of the following form
\begin{equation*}
Q_{\varepsilon}(u)=\Big(\int_{\mathbb{R}^3\backslash\Lambda/\varepsilon}u^2\,{\rm d}x-\varepsilon\Big)_{+}^2.
\end{equation*}

In this way, we can achieve the main result as follows.
\begin{theorem}\label{thm1-1}
Let $2s+2t>3$, $s,t\in(0,1)$. Suppose that $V$ satisfies $(V_0)$, $(V_1)$ and $g\in C^1(\mathbb{R}^{+},\mathbb{R})$ satisfies $(g_0)$--$(g_3)$. Then there exists an $\varepsilon_0>0$ such that system \eqref{main} possesses a positive solution $(u_{\varepsilon},\phi_{\varepsilon})\in H_{\varepsilon}\times\mathcal{D}^{t,2}(\mathbb{R}^3)$ for all $\varepsilon\in(0,\varepsilon_0)$. Moreover, there exists a maximum point $x_{\varepsilon}$ of $u_{\varepsilon}$ such that $\lim\limits_{\varepsilon\rightarrow0}{\rm dist}(x_{\varepsilon},\mathcal{M})=0$ and
\begin{equation*}
u_{\varepsilon}(x)\leq\frac{C\varepsilon^{3+2s}}{C_0\varepsilon^{3+2s}+|x-x_{\varepsilon}|^{3+2s}}\quad x\in\mathbb{R}^3,\,\,\text{and}\,\, \varepsilon\in(0,\varepsilon_0)
\end{equation*}
for some constants $C>0$ and $C_0\in\mathbb{R}$.
\end{theorem}
We give some remarks on the above Theorem.
\begin{remark}
Observe that if $s=t=1$, $\frac{4s+2t}{s+t}=3$, so from $(g_2)$ and $(g_3)$, we see that our assumptions are very weaker than \eqref{equ1-2} and \eqref{equ1-3} in \cite{He} and \cite{WTXZ}, respectively. On the other hand, we consider the local assumption $(V_1)$ comparing the present works appearing in the literature.
\end{remark}

\begin{remark}
If a local $L^{\infty}$-estimate like Theorem 8.17 in \cite{GT} will be established, the assumption $(V_1)$ can be improved as follows
\begin{equation*}
\inf_{\Lambda}V(x)<\inf_{\partial\Lambda}V(x).
\end{equation*}
\end{remark}
The paper is organized as follows, in Section 2, we give some preliminary results. In Section 3, we prove the existence of positive ground state solutions for "limit problem". In Section 4, we prove the main result Theorem \ref{thm1-1}.

\section{Variational Setting}

In this section, we outline the variational framework for studying problem \eqref{main} and list some preliminary Lemma which used later. In the sequel, we denote by $\|\cdot\|_{p}$ the usual norm of the space $L^p(\mathbb{R}^3)$, the letter $c_i$ ($i=1,2,\ldots$) or $C$ denote by some positive constants.

\subsection{Work space stuff}
We define the homogeneous fractional Sobolev space $\mathcal{D}^{\alpha,2}(\mathbb{R}^3)$ as follows
\begin{equation*}
\mathcal{D}^{s,2}(\mathbb{R}^3)=\Big\{u\in L^{2_{s}^{\ast}}(\mathbb{R}^3)\,\,\Big|\,\,|\xi|^{s}(\mathcal{F}u)(\xi)\in L^2(\mathbb{R}^3)\Big\}
\end{equation*}
which is the completion of $C_0^{\infty}(\mathbb{R}^3)$ under the norm
\begin{equation*}
\|u\|_{\mathcal{D}^{s,2}}=\Big(\int_{\mathbb{R}^3}|(-\Delta)^{\frac{s}{2}}u|^2\,{\rm d}x\Big)^{\frac{1}{2}}=\Big(\int_{\mathbb{R}^3}|\xi|^{2s}|(\mathcal{F}u)(\xi)|^2\,{\rm d}\xi\Big)^{\frac{1}{2}}
\end{equation*}
The fractional Sobolev space $H^{s}(\mathbb{R}^3)$ can be described by means of the Fourier transform, i.e.
\begin{equation*}
H^{s}(\mathbb{R}^3)=\Big\{u\in L^2(\mathbb{R}^3)\,\,\Big|\,\,\int_{\mathbb{R}^3}(|\xi|^{2s}|(\mathcal{F}u)(\xi)|^2+|(\mathcal{F}u)(\xi)|^2)\,{\rm d}\xi<+\infty\Big\}.
\end{equation*}
In this case, the inner product and the norm are defined as
\begin{equation*}
(u,v)=\int_{\mathbb{R}^3}(|\xi|^{2s}(\mathcal{F}u)(\xi)\overline{(\mathcal{F}v)(\xi)}+(\mathcal{F}u)(\xi)\overline{(\mathcal{F}v)(\xi)})\,{\rm d}\xi
\end{equation*}
and
\begin{equation*}
 \|u\|_{H^{s}}=\bigg(\int_{\mathbb{R}^3}(|\xi|^{2s}|(\mathcal{F}u)(\xi)|^2+|(\mathcal{F}u)(\xi)|^2)\,{\rm d}\xi\bigg)^{\frac{1}{2}}.
\end{equation*}
From Plancherel's theorem we have $\|u\|_2=\|\mathcal{F}u\|_2$ and $\||\xi|^{s}\mathcal{F}u\|_2=\|(-\Delta)^{\frac{s}{2}}u\|_2$. Hence, the norm of $H^s(\mathbb{R}^3)$ is equivalent to the following
\begin{equation*}
\|u\|_{H^{s}}=\bigg(\int_{\mathbb{R}^3}(|(-\Delta)^{\frac{s}{2}}u(x)|^2+|u(x)|^2)\,{\rm d}x\bigg)^{\frac{1}{2}},\quad \forall u\in H^{s}(\mathbb{R}^3).
\end{equation*}
We denote $\|\cdot\|$ by $\|\cdot\|_{H^{s}}$ in the sequel for convenience.

In terms of finite differences, the fractional Sobolev space $H^{s}(\mathbb{R}^3)$ also can be defined as follows
\begin{equation*}
H^{s}(\mathbb{R}^3)=\Big\{u\in L^2(\mathbb{R}^3)\,\,\Big|\,\,D_{s}u\in L^2(\mathbb{R}^3)\Big\},\quad |D_{s}u|^2=\int_{\mathbb{R}^3}\frac{|u(x)-u(y)|^2}{|x-y|^{3+2s}}\,{\rm d}y
\end{equation*}
endowed with the natural norm
\begin{equation*}
\|u\|_{H^{s}}=\bigg(\int_{\mathbb{R}^3}|u|^2\,{\rm d}x+\int_{\mathbb{R}^3}|D_{s}u|^2\,{\rm d}x\bigg)^{\frac{1}{2}}.
\end{equation*}

Also, in view of Proposition 3.4 and Proposition 3.6 in \cite{NPV}, we have
\begin{equation}\label{equ2-1}
\|(-\Delta)^{\frac{s}{2}}u\|_2^2=\int_{\mathbb{R}^3}|\xi|^{2s}|(\mathcal{F}u)(\xi)|^2\,{\rm d}\xi=\frac{C_{s}}{2}\int_{\mathbb{R}^3}|D_{\alpha}u|^2\,{\rm d}x.
\end{equation}

We define the Sobolev space $H_{\varepsilon}=\{u\in H^s(\mathbb{R}^3)\,\,|\,\, \int_{\mathbb{R}^3}V(\varepsilon x)u^2\,{\rm d}x<\infty\}$ endowed with the norm
\begin{equation*}
\|u\|_{H_{\varepsilon}}=\Big(\int_{\mathbb{R}^3}(|D_su|^2+V(\varepsilon x)u^2)\,{\rm d}x\Big)^{\frac{1}{2}},
\end{equation*}
where we have omitted the constant $\frac{C_s}{2}$ in the front of $|D_su|^2$.

It is well known that $H^s(\mathbb{R}^3)$ is continuously embedded into $L^r(\mathbb{R}^3)$ for $2\leq r\leq 2_{s}^{\ast}$ ($2_{s}^{\ast}=\frac{6}{3-2s}$).
Obviously, this conclusion also holds true for $H_{\varepsilon}$.

\subsection{Formulation of Problem \eqref{main}}
It is easily seen that, just performing the change of variables $u(x)\rightarrow u(x/\varepsilon)$ and $\phi(x)\rightarrow \phi(x/\varepsilon)$, and taking $z=x/\varepsilon$, problem \eqref{main} can be rewritten as the following equivalent form
\begin{equation}\label{main-2-1}
\left\{
  \begin{array}{ll}
   (-\Delta)^su+V(\varepsilon z)u+\phi u=g(u) & \hbox{in $\mathbb{R}^3$,} \\
    (-\Delta)^t\phi=u^2,\,\, u>0& \hbox{in $\mathbb{R}^3$.}
  \end{array}
\right.
\end{equation}
Observe that if $4s+2t\geq3$, there holds $2\leq\frac{12}{3+2t}\leq\frac{6}{3-2s}$ and thus $H_{\varepsilon}\hookrightarrow L^{\frac{12}{3+2t}}(\mathbb{R}^3)$. Considering $u\in H_{\varepsilon}$, the linear functional $\widetilde{\mathcal{L}}_u:\mathcal{D}^{t,2}(\mathbb{R}^3)\rightarrow\mathbb{R}$ is defined by $\widetilde{\mathcal{L}}_u(v)=\int_{\mathbb{R}^3}u^2v\,{\rm d}x$. Similarly, using the Lax-Milgram theorem, there exists a unique $\phi_u^t\in\mathcal{D}^{t,2}(\mathbb{R}^3)$ such that
\begin{align*}
\frac{C_t}{2}\int_{\mathbb{R}^3\times\mathbb{R}^3}\frac{(\phi_u^t(z)-\phi_u^t(y))(v(z)-v(y))}{|z-y|^{3+2s}}\,{\rm d}y\,{\rm d}z&=\int_{\mathbb{R}^3}(-\Delta)^{\frac{t}{2}}\phi_u^t(-\Delta)^{\frac{t}{2}}v\,{\rm d}z\\
&=\int_{\mathbb{R}^3}u^2v\,{\rm d}z,\quad \forall v\in\mathcal{D}^{t,2}(\mathbb{R}^3),
\end{align*}
that is $\phi_u^t$ is a weak solution of $(-\Delta)^t\phi_u^t=u^2$ and so the representation formula holds
\begin{equation*}
\phi_u^t(x)=c_t\int_{\mathbb{R}^3}\frac{u^2(y)}{|x-y|^{3-2t}}\,{\rm d}y,\quad x\in\mathbb{R}^3,\quad c_t=\pi^{-\frac{3}{2}}2^{-2t}\frac{\Gamma(\frac{3-2t}{2})}{\Gamma(t)}.
\end{equation*}
Substituting $\phi_u^t$ in \eqref{main-2-1}, it reduces to a single fractional Schr\"{o}dinger equation
\begin{equation}\label{R-1}
(-\Delta)^su+V(\varepsilon z)u+\phi_u^tu=g(u)\quad z\in\mathbb{R}^3.
\end{equation}
The solvation of \eqref{R-1} can be found by the critical points of the associated energy functional $J_{\varepsilon}: H_{\varepsilon}\rightarrow\mathbb{R}$ defined by
\begin{align*}
J_{\varepsilon}(u)&=\frac{1}{2}\int_{\mathbb{R}^3}|D_su|^2\,{\rm d}z+\frac{1}{2}\int_{\mathbb{R}^3}V(\varepsilon z)u^2\,{\rm d}z+\frac{1}{4}\int_{\mathbb{R}^3}\phi_u^tu^2\,{\rm d}z-\int_{\mathbb{R}^3}G(u)\,{\rm d}z.
\end{align*}

Let us summarize some properties of the function $\phi_u^t$.
\begin{lemma}\label{lem2-1}(\cite{Teng, Teng3})
For every $u\in H_{\varepsilon}$ with $4s+2t\geq3$, define $\Phi(u)=\phi_u^t\in \mathcal{D}^{t,2}(\mathbb{R}^3)$, where $\phi_u^t$ is the unique solution of equation $(-\Delta)^t\phi=u^2$. Then there hold:\\
$(i)$ If $u_n\rightharpoonup u$ in $H_{\varepsilon}$, then $\Phi(u_n)\rightharpoonup\Phi(u)$ in $\mathcal{D}^{t,2}(\mathbb{R}^3)$;\\
$(ii)$ $\Phi(tu)=t^2\Phi(u)$ for any $t\in\mathbb{R}$;\\
$(iii)$ For $u\in H_{\varepsilon}$, one has
\begin{equation*}
\|\Phi(u)\|_{\mathcal{D}^{t,2}}\leq C\|u\|_{\frac{12}{3+2t}}^2\leq C\|u\|_{H_\varepsilon}^2,\quad \int_{\mathbb{R}^3}\Phi(u)u^2\,{\rm d}x\leq C\|u\|_{\frac{12}{3+2t}}^4\leq C\|u\|_{H_\varepsilon}^4,
\end{equation*}
where constant $C$ is independent of $u$;\\
$(iv)$ Let $2s+2t>3$, if $u_n\rightharpoonup u$ in $H_{\varepsilon}$ and $u_n\rightarrow u$ a.e. in $\mathbb{R}^3$, then for any $v\in H_{\varepsilon}$,
\begin{equation*}
\int_{\mathbb{R}^3}\phi_{u_n}^tu_nv\,{\rm d}z\rightarrow\int_{\mathbb{R}^3}\phi_{u}^tuv\,{\rm d}z\quad\text{and}\quad\int_{\mathbb{R}^3}g(u_n)v\,{\rm d}z\rightarrow\int_{\mathbb{R}^3}g(u)v\,{\rm d}z
\end{equation*}
and thus $u$ is a solution for problem \eqref{R-1}.
\end{lemma}

In the end, we recall some regularity results which will be used in the sequel.

\begin{lemma}\label{lem2-1-0}(\cite{Teng3})
Assume that $u_n$ are nonnegative weak solution of
\begin{equation*}
\left\{
  \begin{array}{ll}
   (-\Delta)^su+V_n(x) u+ \phi u=f_n(x,u) & \hbox{in $\mathbb{R}^3$,} \\
    (-\Delta)^t\phi=u^2& \hbox{in $\mathbb{R}^3$,}
  \end{array}
\right.
\end{equation*}
where $\{V_n\}$ satisfies $V_n(x)\geq \alpha_0>0$ for all $x\in\mathbb{R}^3$ and $f_n(x,\tau)$ is a Carathedory function satisfying that for any $\delta>0$, there exists $C_{\delta}>0$ such that
\begin{equation*}
|f_n(x,\tau)|\leq\delta|\tau|+C_{\delta}|\tau|^{2_s^{\ast}-1},\quad \forall (x,\tau)\in\mathbb{R}^3\times\mathbb{R}.
\end{equation*}
Suppose that $u_n$ convergence strongly in $H^s(\mathbb{R}^3)$. Then there exists $C>0$ such that
\begin{equation*}
\|u_n\|_{L^{\infty}}\leq C\quad \text{for all}\,\, n.
\end{equation*}
\end{lemma}

\begin{lemma}\label{lem2-2}(\cite{S})
Let $w=(-\Delta)^su$. Assume $w\in L^{\infty}(\mathbb{R}^n)$ and $u\in L^{\infty}(\mathbb{R}^n)$ for $s>0$.\\
If $2s\leq1$, then $u\in C^{0,\alpha}(\mathbb{R}^n)$ for any $\alpha\leq2s$. Moreover
\begin{equation*}
\|u\|_{C^{0,\alpha}(\mathbb{R}^n)}\leq C\Big(\|u\|_{L^{\infty}(\mathbb{R}^n)}+\|w\|_{L^{\infty}(\mathbb{R}^n)}\Big)
\end{equation*}
for some constant $C$ depending only on $n$, $\alpha$ and $s$.\\
If $2s>1$, then $u\in C^{1,\alpha}(\mathbb{R}^n)$ for any $\alpha<2s-1$. Moreover
\begin{equation*}
\|u\|_{C^{1,\alpha}(\mathbb{R}^n)}\leq C\Big(\|u\|_{L^{\infty}(\mathbb{R}^n)}+\|w\|_{L^{\infty}(\mathbb{R}^n)}\Big)
\end{equation*}
for some constant $C$ depending only on $n$, $\alpha$ and $s$.
\end{lemma}

\begin{lemma}\label{lem2-3}(\cite{Sec})
Assume that $\{u_n\}$ is bounded in $H^s(\mathbb{R}^N)$ and it satisfies
\begin{equation*}
\lim_{n\rightarrow+\infty}\sup_{y\in\mathbb{R}^N}\int_{B_R(y)}|u_n(x)|^2\,{\rm d}x=0
\end{equation*}
where $R>0$. Then $u_n\rightarrow0$ in $L^r(\mathbb{R}^N)$ for every $2<r<2_s^{\ast}$.
\end{lemma}

\section{Limiting problem}

In this section, we consider the "limiting problem" associated with problem \eqref{main-2-1}
\begin{equation}\label{equ3-1}
\left\{
  \begin{array}{ll}
    (-\Delta)^su+\mu u+\phi u=g(u) & \hbox{in $\mathbb{R}^3$,} \\
    (-\Delta)^t\phi=u^2,\,\, u>0& \hbox{in $\mathbb{R}^3$}
  \end{array}
\right.
\end{equation}
for $\mu>0$. We define the energy functional for the limiting problem \eqref{equ3-1} by
\begin{align*}
\mathcal{I}_{\mu}(u)&=\frac{1}{2}\int_{\mathbb{R}^3}|D_su|^2\,{\rm d}x+\frac{\mu}{2}\int_{\mathbb{R}^3}|u|^2\,{\rm d}x+\frac{1}{4}\int_{\mathbb{R}^3}\phi_u^tu^2\,{\rm d}x-\int_{\mathbb{R}^3}G(u)\,{\rm d}x\quad u\in H^s(\mathbb{R}^3).
\end{align*}
Let
\begin{align*}
\mathcal{P}_{\mu}(u)=&\frac{3-2s}{2}\int_{\mathbb{R}^3}|D_su|^2\,{\rm d}x+\frac{3}{2}\int_{\mathbb{R}^3}\mu|u|^2\,{\rm d}x+\frac{3+2t}{4}\int_{\mathbb{R}^3}\phi_{u}^tu^2\,{\rm d}x-3\int_{\mathbb{R}^3}G(u)\,{\rm d}x
\end{align*}
and
\begin{align*}
\mathcal{G}_{\mu}(u)&=(s+t)\langle \mathcal{I}'_{\mu}(u),u\rangle-\mathcal{P}_{\mu}(u)=\frac{4s+2t-3}{2}\int_{\mathbb{R}^3}|D_su|^2\,{\rm d}x+\frac{2s+2t-3}{2}\mu\int_{\mathbb{R}^3}|u|^2\,{\rm d}x\\
&+\frac{4s+2t-3}{4}\int_{\mathbb{R}^3}\phi_u^tu^2\,{\rm d}x+\int_{\mathbb{R}^3}\Big(3G(u)-(s+t)g(u)u\Big)\,{\rm d}x.
\end{align*}
We define the Nehari-Pohozaev manifold
\begin{equation*}
\mathcal{M}_{\mu}=\{u\in H^s(\mathbb{R}^3)\backslash\{0\}\,\,\Big|\,\, \mathcal{G}_{\mu}(u)=0\}
\end{equation*}
and set $b_{\mu}=\inf\limits_{u\in\mathcal{M}_{\mu}}\mathcal{I}_{\mu}(u)$. We list some properties of the manifold $\mathcal{M}_{\mu}$.
\begin{proposition}\label{pro3-2}
The set $\mathcal{M}_{\mu}$ possesses the following properties:\\
$(i)$ $0\not\in\partial\mathcal{M}_{\mu}$;\\
$(ii)$ for any $u\in H^s(\mathbb{R}^3)\backslash\{0\}$, there exists a unique $\tau_0:=\tau(u)>0$ such that $u_{\tau_0}\in\mathcal{M}_{\mu}$, where $u_{\tau}=\tau^{s+t}u(\tau x)$. Moreover,
\begin{equation*}
\mathcal{I}_{\mu}(u_{\tau_0})=\max_{\tau\geq0}\mathcal{I}_{\mu}(u_{\tau});
\end{equation*}
\end{proposition}

\begin{proof}
The proof of $(i)$ and $(ii)$ is standard, it is only to prove the uniqueness of $\tau$ of $(ii)$. Indeed, if there exist $\tau_1>\tau>0$ such that $u_{\tau_1},u_{\tau}\in\mathcal{M}_{\mu}$, then
\begin{equation*}
\mathcal{G}_{\mu}(u_{\tau_1})=0,\quad \mathcal{G}_{\mu}(u_{\tau})=0.
\end{equation*}
By simple computation, we have
\begin{align*}
&\frac{2s+2t-3}{2}(\frac{1}{\tau_1^{2s}}-\frac{1}{\tau^{2s}})\int_{\mathbb{R}^3}\mu|u|^2\,{\rm d}x\\
&=\int_{\mathbb{R}^3}\Big(\frac{(s+t)g(u_{\tau_1})u_{\tau_1}-3G(u_{\tau_1})}{\tau_1^{4s+2t-3}}-\frac{(s+t)g(u_{\tau})u_{\tau}-3G(u_{\tau})}{\tau^{4s+2t-3}}\Big)\,{\rm d}x\\
&+\frac{2_s^{\ast}(s+t)-3}{2_s^{\ast}}(\tau_1^{(2_s^{\ast}-4)s+(2_s^{\ast}-2)t}-\tau^{(2_s^{\ast}-4)s+(2_s^{\ast}-2)t})\int_{\mathbb{R}^3}(u^{+})^{2_s^{\ast}}\,{\rm d}x\\
&=\int_{\mathbb{R}^3}\Big(\frac{(s+t)g(\tau_1^{s+t}u)\tau_1^{s+t}u-3G(\tau_1^{s+t}u)}{\tau_1^{4s+2t}}-\frac{(s+t)g(\tau^{s+t}u)\tau^{s+t}u-3G(\tau^{s+t}u)}{\tau^{4s+2t}}\Big)\,{\rm d}x.
\end{align*}
If we show that the function $\tau\in\mathbb{R}^{+}\rightarrow\frac{(s+t)g(\tau^{s+t}u)\tau^{s+t}u-3G(\tau^{s+t}u)}{\tau^{4s+2t}}$ is non-decreasing, then we get a contradiction and the uniqueness is proved. In fact, by computation and using $(g_3)$, we deduce that
\begin{align*}
&\Big(\frac{(s+t)g(\tau^{s+t}u)\tau^{s+t}u-3G(\tau^{s+t}u)}{\tau^{4s+2t}}\Big)'=\frac{1}{\tau^{4s+2t+1}}\Big(3(4s+2t)G(\tau^{s+t}u)\\
&-(s+t)(3s+t+3)g(\tau^{s+t}u)\tau^{s+t}u+(s+t)^2g'(\tau^{s+t}u)\tau^{2(s+t)}u^2\Big)\\
&\geq\frac{1}{\tau^{4s+2t+1}}\Big[\Big((s+t)^2(q-1)-(s+t)(3s+t+3)\Big)g(\tau^{s+t}u)\tau^{s+t}u+3(4s+2t)G(\tau^{s+t}u)\Big]\\
&>0.
\end{align*}
\end{proof}

\begin{lemma}\label{lem3-1}
$\mathcal{I}_{\mu}$ possesses the mountain pass geometry:\\
$(i)$ there exist $\rho_0,\beta_0>0$ such that $\mathcal{I}_{\mu}(u)\geq\beta_0$ for all $u\in H^s(\mathbb{R}^3)$ with $\|u\|=\rho_0$;\\
$(ii)$ there exists $u_0\in H^s(\mathbb{R}^3)$ such that $\mathcal{I}_{\mu}(u_0)<0$.
\end{lemma}

\begin{proof}
By $(g_0)$ and $(g_1)$, for any $\eta>0$, there exists $C_{\eta}>0$ such that
\begin{equation}\label{equ3-4}
g(t)\leq\eta|t|+C_{\eta}|t|^{2_s^{\ast}-1}\quad \text{and}\quad G(t)\leq\frac{\eta}{2}|t|^2+C_{\eta}|t|^{2_s^{\ast}}\quad \text{for any}\,\, t\in\mathbb{R}.
\end{equation}
Hence, choosing $\eta=\frac{\mu}{2}$ and by Sobolev inequality, we have that
\begin{align*}
\mathcal{I}_{\mu}(u)&\geq\frac{1}{2}\int_{\mathbb{R}^3}|D_su|^2\,{\rm d}x+\frac{\mu}{2}\int_{\mathbb{R}^3}|u|^2\,{\rm d}x-\frac{\eta}{2}\int_{\mathbb{R}^3}|u|^2\,{\rm d}x-C_{\eta}\int_{\mathbb{R}^3}|u|^{2_s^{\ast}}\,{\rm d}x\\
&\geq\frac{1}{4}\|u\|^2-C_{\mu}\|u\|^{2_s^{\ast}},
\end{align*}
thus, there exists $\rho_0,\beta_0>0$ small enough such that $\mathcal{I}_{\mu}(u)\geq\beta_0$ for $\|u\|=\rho_0$.

$(ii)$ For any $u\in H^s(\mathbb{R}^3)$ with $u\geq0$, set $u_{\tau}(x)=\tau^{(s+t)}u(\tau x)$ with $\tau>0$. Thus, by $(f_3)$, we deduce that
\begin{align*}
\mathcal{I}_{\mu}(u_{\tau})&\leq\frac{\tau^{(4s+2t-3)}}{2}\int_{\mathbb{R}^3}|D_su|^2\,{\rm d}x+\frac{\tau^{(2s+2t-3)}}{2}\int_{\mathbb{R}^3}\mu|u|^2\,{\rm d}x\\
&+\frac{\tau^{(4s+2t-3)}}{4}\int_{\mathbb{R}^3}\phi_u^tu^2\,{\rm d}x-C\tau^{q(s+t)-3}\int_{\mathbb{R}^3}|u|^q\,{\rm d}x.
\end{align*}
Since $4s+2t>3$ and so $4s+2t-3<q(s+t)-3$, we obtain that $\mathcal{I}_{\mu}(u_{\tau})\rightarrow-\infty$ as $\tau\rightarrow+\infty$. Hence, there exists $\tau_0>0$ large enough such that $\mathcal{I}_{\mu}(u_0)<0$, where $u_0=u_{\tau_0}$.
\end{proof}

From Lemma \ref{lem3-1}, we can define the mountain-pass level of $\mathcal{I}_{\mu}$ as follows
\begin{equation*}
c_{\mu}=\inf_{\gamma\in\Gamma_{\mu}}\sup_{t\in[0,1]}\mathcal{I}_{\mu}(\gamma(t))
\end{equation*}
where
\begin{equation*}
\Gamma_{\mu}=\Big\{\gamma\in C([0,1],H^s(\mathbb{R}^3))\,\,\Big|\,\, \gamma(0)=0,\,\, \mathcal{I}_{\mu}(\gamma(1))<0\Big\}
\end{equation*}
and $c_{\mu}>0$. By the condition $(f_3)$ and using Lemma \ref{lem3-1}, we can show the equivalent characterization of mountain-pass level $c_{\mu}$.
\begin{lemma}\label{lem3-2}
\begin{equation*}
c_{\mu}=b_{\mu}.
\end{equation*}
\end{lemma}
\begin{proof}
We only need to verify that $\gamma([0,1])\cap\mathcal{M}_{\mu}\neq\emptyset$. Indeed, by Lemma \ref{lem3-1}, we see that if $u\in H^s(\mathbb{R}^3)\backslash\{0\}$, is interior to or on $\mathcal{M}_{\mu}$, then
\begin{align*}
&\frac{4s+2t-3}{2}\int_{\mathbb{R}^3}|D_su|^2\,{\rm d}x+\frac{2s+2t-3}{2}\int_{\mathbb{R}^3}\mu|u|^2\,{\rm d}x+\frac{4s+2t-3}{4}\int_{\mathbb{R}^3}\phi_{u}^tu^2\,{\rm d}x\\
&\geq\int_{\mathbb{R}^3}\Big((s+t)g(u)u-3G(u)\Big)\,{\rm d}x
\end{align*}
and
\begin{align*}
(4s+2t-3)\mathcal{I}_{\mu}(u)&=\mathcal{G}_{\mu}(u)+s\int_{\mathbb{R}^3}\mu|u|^2\,{\rm d}x+\int_{\mathbb{R}^3}\Big((s+t)g(u)u-(4s+2t)G(u)\Big)\,{\rm d}x>0.
\end{align*}
Hence $\gamma$ crosses $\mathcal{M}_{\mu}$ since $\gamma(0)=0$, $\mathcal{I}_{\mu}(\gamma(1))<0$ which implies that $\mathcal{G}_{\mu}(\gamma(1))<0$, combining with $\mathcal{G}_{\mu}(\gamma(t))\geq0$. Therefore,
\begin{equation*}
\max_{t\in[0,1]}\mathcal{I}_{\mu}(\gamma(t))\geq\inf_{\mathcal{M}_{\mu}}\mathcal{I}_{\mu}(w)=b_{\mu}
\end{equation*}
and then $c_{\mu}\geq b_{\mu}$.
\end{proof}
In order to obtain the boundedness of $(PS)$ sequence, we will construct a $(PS)$ sequence $\{u_n\}$ for $\mathcal{I}_{\mu}$ at the level $c_{\mu}$ that satisfies $\mathcal{G}_{\mu}(u_n)\rightarrow0$ as $n\rightarrow+\infty$ i.e.,
\begin{lemma}\label{lem3-3}
There exists a sequence $\{u_n\}$ in $H^s(\mathbb{R}^3)$ such that as $n\rightarrow+\infty$,
\begin{equation}\label{equ3-5}
\mathcal{I}_{\mu}(u_n)\rightarrow c_{\mu},\quad\mathcal{I}'_{\mu}(u_n)\rightarrow0,\quad \mathcal{G}_{\mu}(u_n)\rightarrow0.
\end{equation}
\end{lemma}

\begin{proof}
Define the map $\Phi:\mathbb{R}\times H^s(\mathbb{R}^3)\rightarrow H^s(\mathbb{R}^3)$ for $\theta\in\mathbb{R}$, $v\in H^s(\mathbb{R}^3)$ by
$\Phi(\theta,v)(x)=e^{(s+t)\theta}v(e^{\theta}x)$. By computation, for every $\theta\in\mathbb{R}$, $v\in H^s(\mathbb{R}^3)$, we see that the functional $\mathcal{I}_{\mu}\circ\Phi$ writes as
\begin{align*}
(\mathcal{I}_{\mu}\circ\Phi)(\theta,v)&=\frac{e^{(4s+2t-3)\theta}}{2}\int_{\mathbb{R}^3}|D_sv|^2\,{\rm d}x+\frac{e^{(2s+2t-3)\theta}}{2}\int_{\mathbb{R}^3}|v|^2\,{\rm d}x\\
&+\frac{e^{(4s+2t-3)\theta}}{4}\int_{\mathbb{R}^3}\phi_u^tu^2\,{\rm d}x-e^{3\theta}\int_{\mathbb{R}^3}G(e^{(s+t)}v)\,{\rm d}x.
\end{align*}
Similarly as the proof of $(i)$ of Lemma \ref{lem3-1}, we have that
\begin{equation*}
(\mathcal{I}_{\mu}\circ\Phi)(\theta,v)\geq \frac{1}{4}\|\Phi(\theta,v)\|^2-C\|\Phi(\theta,v)\|^{2_s^{\ast}}.
\end{equation*}
Thus, there exists $\rho_1,\alpha_1>0$ small such that $(\mathcal{I}_{\mu}\circ\Phi)(\theta,v)\geq \alpha_1$ for every $\theta\in\mathbb{R}$ and $v\in H^s(\mathbb{R}^3)$ with $\|\Phi(\theta,v)\|=\rho_1$. Moreover, we have that $(\mathcal{I}_{\mu}\circ\Phi)(0,u_0)<0$, where $w_0$ is given in Lemma \ref{lem3-1}. Hence, $\mathcal{I}_{\mu}\circ\Phi$ possesses the mountain-pass geometry in $\mathbb{R}\times H^s(\mathbb{R}^3)$. We define the mountain-pass level of $\mathcal{I}_{\mu}\circ\Phi$
\begin{equation*}
\widetilde{c_{\mu}}=\inf_{\widetilde{\gamma}\in\widetilde{\Gamma_{\mu}}}\max_{t\in[0,1]}(\mathcal{I}_{\mu}\circ\Phi)(\widetilde{\gamma}(t)),
\end{equation*}
where $\widetilde{\Gamma_{\mu}}=\{\widetilde{\gamma}\in C([0,1],\mathbb{R}\times H^s(\mathbb{R}^3))\,\, |\,\,\widetilde{\gamma}(0)=0,\,\, (\mathcal{I}_{\mu}\circ\Phi)(\widetilde{\gamma}(1))<0 \}$. Observe that $\Gamma_{\mu}=\{\Phi\circ\widetilde{\gamma}\,\,|\,\, \widetilde{\gamma}\in\widetilde{\Gamma_{\mu}}\}$, the mountain-pass level of $\mathcal{I}_{\mu}$ coincides with $\mathcal{I}_{\mu}\circ\Phi$, i.e., $c_{\mu}=\widetilde{c_{\mu}}$.

By the general minimax principle (\cite{Willem}, Theorem 2.8), there exists a sequence $\{(\theta_n,v_n)\}\subset\mathbb{R}\times H^s(\mathbb{R}^3)$ such that
\begin{align}\label{equ3-6}
(\mathcal{I}_{\mu}\circ\Phi)(\theta_n,v_n)\rightarrow c_{\mu},\quad (\mathcal{I}_{\mu}\circ\Phi)'(\theta_n,v_n)\rightarrow0, \quad \theta_n\rightarrow0.
\end{align}
The detailed proof refer the readers to see Proposition 3.4 in \cite{HL1}. For every $(h,\phi)\in\mathbb{R}\times H^s(\mathbb{R}^3)$, we deduce that
\begin{equation}\label{equ3-7}
(\mathcal{I}_{\mu}\circ\Phi(\theta_n,v_n))'(h,\phi)=\langle\mathcal{I}_{\mu}'(\Phi(\theta_n,v_n)),\Phi(\theta_n,\phi)\rangle+\mathcal{G}_{\mu}(\Phi(\theta_n,v_n))h.
\end{equation}
Taking $h=1$, $\phi=0$ in \eqref{equ3-7}, we get
\begin{equation*}
\mathcal{G}_{\mu}(\Phi(\theta_n,v_n))\rightarrow0.
\end{equation*}
For every $\phi\in H^s(\mathbb{R}^3)$, set $\varphi(x)=e^{-(s+t)\theta_n}\phi(e^{-\theta_n}x)$, $h=0$ in \eqref{equ3-7}, by \eqref{equ3-6}, we get
\begin{equation*}
\langle\mathcal{I}_{\mu}'(\Phi(\theta_n,v_n)),\phi\rangle=o_n(1)\|e^{-(s+t)\theta_n}\phi(e^{-\theta_n}x)\|=o_n(1)\|\phi\|.
\end{equation*}
Denoting $u_n=\Phi(\theta_n,v_n)$, combining with \eqref{equ3-6}, the conclusion follows.
\end{proof}

\begin{lemma}\label{lem3-4}
Every sequence $\{u_n\}\subset H^s(\mathbb{R}^3)$ satisfying \eqref{equ3-5} is bounded in $H^s(\mathbb{R}^3)$.
\end{lemma}

\begin{proof}
By \eqref{equ3-5}, we deduce that
\begin{align*}
&c_{\mu}+o_n(1)=\mathcal{I}_{\mu}(u_n)-\frac{1}{q(s+t)-3}\mathcal{G}_{\mu}(u_n)\\
&=\frac{(q-4)s+(q-2)t}{2(q(s+t)-3)}\int_{\mathbb{R}^3}|D_su_n|^2\,{\rm d}x+\frac{(q-2)(s+t)}{2(q(s+t)-3)}\mu\int_{\mathbb{R}^3}|u_n|^2\,{\rm d}x\\
&+\frac{(q-4)s+(q-2)t}{4(q(s+t)-3)}\int_{\mathbb{R}^3}\phi_{u_n}^tu_n^2\,{\rm d}x+\frac{s+t}{q(s+t)-3}\int_{\mathbb{R}^3}\Big(g(u_n)u_n-qG(u_n)\Big)\,{\rm d}x
\end{align*}
which implies the boundedness of the sequence $\{u_n\}$ in $H^s(\mathbb{R}^3)$ due to $q>\frac{4s+2t}{s+t}$.

\end{proof}

By using the Vanishing Lemma \ref{lem2-3}, it is not difficult to deduce that the bounded sequence $\{u_n\}\subset H^s(\mathbb{R}^3)$ given in \eqref{equ3-5} is non-vanishing. That is,
\begin{lemma}\label{lem3-5}
There exists a sequence $\{x_n\}\subset\mathbb{R}^3$ and $R>0$, $\beta>0$ such that $\int_{B_R(x_n)}|u_n|^2\,{\rm d}x\geq\beta$.
\end{lemma}

Combining Lemma \ref{lem3-4} with Lemma \ref{lem3-3} and Lemma \ref{lem3-5}, we can show the existence of positive ground state solution for the limiting problem \eqref{equ3-1}.
\begin{proposition}\label{pro3-3}
Problem \eqref{equ3-1} possesses a positive ground state solution $u\in H^s(\mathbb{R}^3)$.
\end{proposition}
\begin{proof}
Let $\{u_n\}$ be the sequence given in \eqref{equ3-5}. Set $\widetilde{u}_n(x)=u_n(x+x_n)$, where $\{x_n\}$ is the sequence obtained in Lemma \ref{lem3-5}. Thus $\{\widetilde{u}_n\}$ is still bounded in $H^s(\mathbb{R}^3)$ and so up to a subsequence, still denoted by $\{\widetilde{u}_n\}$, we may assume that there exists $\widetilde{u}\in H^s(\mathbb{R}^3)$ such that
\begin{equation*}
\left\{
  \begin{array}{ll}
    \widetilde{u}_n\rightharpoonup \widetilde{u} & \hbox{in $H^s(\mathbb{R}^3)$,} \\
    \widetilde{u}_n\rightarrow \widetilde{u} & \hbox{in $L_{loc}^p(\mathbb{R}^3)$ for all $1\leq p<2_s^{\ast}$,}\\
    \widetilde{u}_n\rightarrow \widetilde{u} & \hbox{a.e. $\mathbb{R}^3$.}
  \end{array}
\right.
\end{equation*}
It follows from Lemma \ref{lem3-5} that $\widetilde{u}$ is nontrivial. Moreover, $\widetilde{u}$ is a nontrivial solution of problem \eqref{equ3-1}, and so $\mathcal{G}_{\mu}(\widetilde{u})=0$. By Fatou's Lemma and \eqref{equ3-5}, we have
\begin{align*}
c_{\mu}&=b_{\mu}\leq\mathcal{I}_{\mu}(\widetilde{u})=\mathcal{I}_{\mu}(\widetilde{u})-\frac{1}{4s+2t-3}\mathcal{G}_{\mu}(\widetilde{u})=\frac{s}{4s+2t-3}\int_{\mathbb{R}^3}\mu|\widetilde{u}|^2\,{\rm d}x\\
&+\frac{s+t}{4s+2t-3}\int_{\mathbb{R}^3}\Big(f(\widetilde{u})\widetilde{u}-\frac{4s+2t}{s+t}F(\widetilde{u})\Big)\,{\rm d}x\\
&\leq\liminf_{n\rightarrow\infty}\Big[\frac{s+t}{4s+2t-3}\int_{\mathbb{R}^3}\Big(g(\widetilde{u}_n)\widetilde{u}_n-\frac{4s+2t}{s+t}G(\widetilde{u}_n)\Big)\,{\rm d}x+\frac{s}{4s+2t-3}\int_{\mathbb{R}^3}\mu|\widetilde{u}_n|^2\,{\rm d}x\Big]\\
&=\liminf_{n\rightarrow\infty}\Big[\mathcal{I}_{\mu}(\widetilde{u}_n)-\frac{1}{4s+2t-3}\mathcal{G}_{\mu}(\widetilde{u}_n)\Big]=\liminf_{n\rightarrow\infty}\Big[\mathcal{I}_{\mu}(u_n)-\frac{1}{4s+2t-3}\mathcal{G}_{\mu}(u_n)\Big]=c_{\mu}
\end{align*}
which implies that $\widetilde{u}_n\rightarrow \widetilde{u}$ in $H^s(\mathbb{R}^3)$. Indeed, from the above inequality, we get that
\begin{equation*}
\int_{\mathbb{R}^3}\widetilde{u}_n^2\,{\rm d}x\rightarrow\int_{\mathbb{R}^3}\widetilde{u}^2\,{\rm d}x.
\end{equation*}
By virtue of the Brezis-Lieb Lemma and interpolation argument, we conclude that
\begin{equation*}
\widetilde{u}_n\rightarrow\widetilde{u}\quad \text{in}\,\, L^r(\mathbb{R}^3)\,\, \text{for all}\,\, 2\leq r< 2_s^{\ast}.
\end{equation*}
Hence, from the standard arguments, it follows that $\widetilde{u}_n\rightarrow \widetilde{u}$ in $H^s(\mathbb{R}^3)$. Therefore, by Lemma \ref{lem3-2}, we conclude that $\mathcal{I}_{\mu}(\widetilde{u})=c_{\mu}$ and $\mathcal{I}'_{\mu}(\widetilde{u})=0$.

Next, we show that the ground state solution of \eqref{equ3-1} is positive. Indeed, by standard argument to the proof Proposition 4.4 in \cite{Teng2}, using Lemma \ref{lem2-2} two times and the hypothesis $(g_1)$, we have that $\widetilde{u}\in C^{2,\alpha}(\mathbb{R}^3)$ for some $\alpha\in(0,1)$ for $s>\frac{1}{2}$. Using $-\widetilde{u}^{-}$ as a testing function, it is easy to see that $\widetilde{u}\geq0$. Since $\widetilde{u}\in C^{2,\alpha}(\mathbb{R}^3)$, by Lemma 3.2 in \cite{NPV}, we have that
\begin{equation*}
(-\Delta)^s\widetilde{u}(x)=-\frac{C_s}{2}\int_{\mathbb{R}^3}\frac{\widetilde{u}(x+y)+\widetilde{u}(x-y)-2\widetilde{u}(x)}{|x-y|^{3+2s}}\,{\rm d}x\,{\rm d}y,\quad \forall\,\, x\in\mathbb{R}^3.
\end{equation*}
Assume that there exists $x_0\in\mathbb{R}^3$ such that $\widetilde{u}(x_0)=0$, then from $\widetilde{u}\geq0$ and $\widetilde{u}\not\equiv0$, we get
\begin{equation*}
(-\Delta)^s\widetilde{u}(x_0)=-\frac{C_s}{2}\int_{\mathbb{R}^3}\frac{\widetilde{u}(x_0+y)+\widetilde{u}(x_0-y)}{|x_0-y|^{3+2s}}\,{\rm d}x\,{\rm d}y<0.
\end{equation*}
However, observe that $(-\Delta)^s\widetilde{u}(x_0)=-\mu \widetilde{u}(x_0)-(\phi_{\widetilde{u}}^t\widetilde{u})(x_0)+f(\widetilde{u}(x_0))+\widetilde{u}(x_0)^{2_s^{\ast}-1}=0$, a contradiction. Hence, $\widetilde{u}(x)>0$, for every $x\in\mathbb{R}^3$. The proof is completed.

\end{proof}

Let $\mathcal{L}_{\mu}$ be the set of ground state solutions $W$ of \eqref{equ3-1} satisfying $W(0)=\max\limits_{\mathbb{R}^3}W(x)$. Then we obtain the following compactness of $\mathcal{L}_{\mu}$.

\begin{proposition}\label{pro3-4}
$(i)$ For each $\mu>0$, $\mathcal{L}_{\mu}$ is compact in $H^s(\mathbb{R}^3)$.\\
$(ii)$ $0<W(x)\leq\frac{C}{1+|x|^{3+2s}}$ for any $x\in\mathbb{R}^3$.
\end{proposition}

\begin{proof}
$(i)$  For any $W\in H^s(\mathbb{R}^3)$, we have
\begin{align*}
c_{\mu}&=\mathcal{I}_{\mu}(W)-\frac{1}{q(s+t)-3}\mathcal{G}_{\mu}(W)\\
&=\frac{(q-4)s+(q-2)t}{2(q(s+t)-3)}\int_{\mathbb{R}^3}|D_s W|^2\,{\rm d}x+\frac{(q-2)(s+t)}{2(q(s+t)-3)}\mu\int_{\mathbb{R}^3}W^2\,{\rm d}x\\
&+\frac{(q-4)s+(q-2)t}{4(q(s+t)-3)}\int_{\mathbb{R}^3}\phi_{W}^tW^2\,{\rm d}x+\frac{s+t}{q(s+t)-3}\int_{\mathbb{R}^3}\Big(g(W)W-qG(W)\Big)\,{\rm d}x
\end{align*}
which yields the boundedness of $\mathcal{L}{_\mu}$ in $H^s(\mathbb{R}^3)$.

Similar to the proof of Lemma \ref{lem3-5} and Proposition \ref{pro3-3}, we verify that for any bounded $\{W_n\}\subset\mathcal{L}_{\mu}$, up to a subsequence, there exist $\{x_n\}\subset \mathbb{R}^3$ and $\overline{W}_0\in H^s(\mathbb{R}^3)$ such that $\overline{W}_n(x):=W_n(x+x_n)\rightarrow \overline{W}_0$ in $H^s(\mathbb{R}^3)$. By Lemma \ref{lem2-1-0}, we see that
\begin{equation}\label{equ3-8}
\|W_n\|_{\infty}=\|\overline{W}_n\|_{\infty}\leq C,
\end{equation}
where $C$ is independent on $n$.

On the other hand, from the boundedness of $\{W_n\}$ in $H^s(\mathbb{R}^3)$, up to a subsequence, we may assume that there exists $W_0\in H^s(\mathbb{R}^3)$ such that $W_n\rightharpoonup W_0$ in $H^s(\mathbb{R}^3)$ and $W_n\rightarrow W_0$ in $L_{loc}^r(\mathbb{R}^3)$ for $1\leq r<2_s^{\ast}$ and $W_n\rightarrow W_0$ a.e. $\mathbb{R}^3$. Since $W_n$ is a solution of \eqref{equ3-1}, in view of Lemma \ref{lem2-2} and \eqref{equ3-8}, we see that $\|W_n\|_{C^{1,\alpha}(\mathbb{R}^3)}\leq C$ for some $\alpha\in(0,1)$, where $C$ depending only on $\alpha$ and $s$.  The Arzela-Ascoli's Theorem shows that $W_n(0)\rightarrow W_0(0)$ as $n\rightarrow\infty$. Since $W_n(0)$ is a global maximum for $W_n(x)$, then we have that
\begin{equation*}
0\leq(-\Delta)^sW_n(0)=-\mu W_n(0)-\phi_{W_n}^t(0)W_n(0)+g(W_n(0))
\end{equation*}
which leads to $W_n(0)\geq C_0>0$. Hence, $W_0(0)\geq C_0>0$, this means that $W_0$ is nontrivial.

Finally, similar arguments as in the proof of Proposition \ref{pro3-3}, we can show that $W_n\rightarrow W_0$ in $H^s(\mathbb{R}^3)$. This completes the proof that $\mathcal{L}_{\mu}$ is compact in $H^s(\mathbb{R}^3)$.

$(ii)$  By Lemma 4.2 and Lemma 4.3 in \cite{FQT}, by scaling, there exists a continuous function $U$ such that
\begin{equation*}
0<U(x)\leq \frac{C}{1+|x|^{3+2s}}
\end{equation*}
and
\begin{equation*}
(-\Delta)^sU+\frac{\mu}{2}U=0 \quad \text{on}\,\,\mathbb{R}^3\backslash B_R(0).
\end{equation*}
for some suitable $R>0$. By standard argument, using the fact that $W\in L^p(\mathbb{R}^3)\cap C^{1,\alpha}(\mathbb{R}^3)$ for all $2\leq p\leq\infty$, we infer that $\lim\limits_{|x|\rightarrow\infty}W(x)=0$. Thus there exists $R_1>0$ (we can choose $R_1>R$) large enough such that
\begin{align*}
(-\Delta)^sW+\frac{\mu}{2}W&=(-\Delta)^sW+\mu W-\frac{\mu}{2} W=g(W)-\phi_{W}^tW-\frac{\mu}{2}W\\
&\leq g(W)-\frac{\mu}{2}W\leq0
\end{align*}
for any $x\in\mathbb{R}^3\backslash B_{R_1}(0)$. Therefore, we have obtained that
\begin{equation}\label{equ3-11}
(-\Delta)^sU+\frac{\mu}{2}U\geq(-\Delta)^sW+\frac{\mu}{2}W\quad \text{on}\,\, \mathbb{R}^3\backslash B_{R_1}(0).
\end{equation}

Let $\mathbb{A}=\inf\limits_{B_{R_1}(0)}U>0$, $Z(x)=(\mathbb{B}+1)U-\mathbb{A}W$, where $\mathbb{B}=\|W\|_{\infty}\leq C<\infty$. We claim that $Z(x)\geq0$ for all $x\in\mathbb{R}^3$. If the claim is true, we have that
\begin{equation*}
0<W(x)\leq\frac{\mathbb{B}+1}{\mathbb{A}}U(x)\leq \frac{C}{1+|x|^{3+2s}}\quad \text{for all}\,\, x\in\mathbb{R}^3
\end{equation*}
and the conclusion is proved.

Suppose by contradiction that there exists $\{x_n\}\subset\mathbb{R}^3$ such that
\begin{equation}\label{equ3-12}
\inf_{x\in\mathbb{R}^3}Z(x)=\lim_{n\rightarrow\infty}Z(x_n)<0.
\end{equation}
Since $\lim\limits_{|x|\rightarrow\infty}U(x)=\lim\limits_{|x|\rightarrow\infty}W(x)=0$ by virtue of \eqref{equ3-11}, then $\lim\limits_{|x|\rightarrow\infty}Z(x)=0$. Hence, sequence $\{x_n\}$ must be bounded and then up to a subsequence, we may assume that $x_n\rightarrow x_0\in\mathbb{R}^3$. From \eqref{equ3-12} and the continuity of $Z(x)$, we have that
\begin{equation*}
\inf_{x\in\mathbb{R}^3}Z(x)=Z(x_0)<0
\end{equation*}
which yields
\begin{align*}
(-\Delta)^sZ(x_0)+\frac{\mu}{2}Z(x_0)&=\frac{\mu}{2}Z(x_0)-\frac{C_s}{2}
\int_{\mathbb{R}^3}\frac{Z(x_0+y)+Z(x_0-y)-2Z(x_0)}{|x-y|^{3+2s}}\,{\rm d}y\\
&<0.
\end{align*}
Note that $Z(x)\geq \mathbb{A}\mathbb{B}+U-\mathbb{A}\mathbb{B}>0$ on $B_{R_1}(0)$, this leads to $x_0\in\mathbb{R}^3\backslash B_{R_1}(0)\subset\mathbb{R}^3\backslash B_R(0)$. From \eqref{equ3-11}, we have that
\begin{align*}
(-\Delta)^sZ(x_0)+\frac{\mu}{2}Z(x_0)&=\Big[(\mathbb{B}+1)\Big((-\Delta)^sU+\frac{\mu}{2}U\Big)-\mathbb{A}\Big((-\Delta)^sW+\frac{\mu}{2}W\Big)\Big]\Big|_{x=x_0}\\
&\geq0
\end{align*}
which is a contradiction. Thus, the claim holds true and the proof is completed.

\end{proof}

\section{The penalization scheme}

For the bounded domain $\Lambda$ given in $(V_1)$, $k>2$, $a>0$ such that $g(a)=\frac{V_0}{k}a$ where $V_0$ is defined in $(V_0)$, we consider a new problem
\begin{equation}\label{main-4-1}
(-\Delta)^su+V(\varepsilon z)u+\phi_u^t u=f(\varepsilon z,u) \quad \text{in}\,\,\mathbb{R}^3,
\end{equation}
where $f(\varepsilon z,\tau)=\chi_{\Lambda_{\varepsilon}}(\varepsilon z)g(\tau)+(1-\chi_{\Lambda_{\varepsilon}}(\varepsilon z))\tilde{g}(\tau)$ with
\begin{equation*}
\tilde{g}(\tau)=\left\{
  \begin{array}{ll}
    f(\tau) & \hbox{if $\tau\leq a$,} \\
    \frac{V_0}{k}\tau & \hbox{if $\tau>a$}
  \end{array}
\right.
\end{equation*}
and $\chi_{\Lambda_{\varepsilon}}(\varepsilon z)=1$ if $z\in\Lambda_{\varepsilon}$, $\chi(z)=0$ if $z\not\in\Lambda_{\varepsilon}$, where $\Lambda_{\varepsilon}=\Lambda/\varepsilon$. It is easy to see that under the assumptions $(g_0)$-$(g_3)$, $f(z,\tau)$ is a Caratheodory function and satisfies the following assumptions:\\
$(f_1)$ $f(z,\tau)=o(\tau)$ as $\tau\rightarrow0$ uniformly on $z\in\mathbb{R}^3$;\\
$(f_2)$ $f(z,\tau)\leq g(\tau)$ for all $\tau\in\mathbb{R}^{+}$ and $z\in\mathbb{R}^3$, $f(z,\tau)=0$ for all $z\in\mathbb{R}^3$ and $\tau<0$, $f(z,\tau)=G(\tau)$ for $z\in\mathbb{R}^3$, $\tau\in[0,a]$;\\
$(f_3)$ $0<2\tilde{G}(\tau)\leq\tilde{g}(\tau)\tau\leq\frac{V_0}{k}\tau^2\leq\frac{V(x)}{k}\tau^2$ for all $s\geq0$ with the number $k>2$, where $\tilde{G}(\tau)$ is a prime function of $\tilde{g}$;\\
$(f_4)$ $\frac{f(z,s\tau)}{\tau}$ is nondecreasing in $\tau\in\mathbb{R}^{+}$ uniformly for $z\in\mathbb{R}^3$, $\frac{f(z,s\tau)}{\tau^{q-1}}$ is nondecreasing in $\tau\in\mathbb{R}^{+}$ and $z\in\Lambda$, $\frac{f(z,s\tau)}{\tau^{q-1}}$ is nondecreasing in $\tau\in(0,a)$ and $z\in\mathbb{R}^3\backslash\Lambda$.

Obviously, if $u_{\varepsilon}$ is a solution of \eqref{main-4-1} satisfying $u_{\varepsilon}(z)\leq a$ for $z\in\mathbb{R}^3$, then $u_{\varepsilon}$ is indeed a solution of the original problem \eqref{R-1}.

For $u\in H_{\varepsilon}$, let
\begin{equation*}
P_{\varepsilon}(u)=\frac{1}{2}\int_{\mathbb{R}^3}(|D_su|^2+V(\varepsilon z)u^2)\,{\rm d}z+\frac{1}{4}\int_{\mathbb{R}^3}\phi_u^tu^2\,{\rm d}z-\int_{\mathbb{R}^3}F(\varepsilon z,u)\,{\rm d}z.
\end{equation*}
We define
\begin{equation*}
Q_{\varepsilon}(v)=\Big(\int_{\mathbb{R}^3\backslash\Lambda_{\varepsilon}}v^2\,{\rm d}z-\varepsilon\Big)_{+}^2.
\end{equation*}
This type of penalization was firstly introduced in \cite{BW}, which will act as a penalization to force the concentration phenomena to occur inside $\Lambda$. Let us define the functional $\mathcal{J}_{\varepsilon}: H_{\varepsilon}\rightarrow\mathbb{R}$ as follows
\begin{equation*}
\mathcal{J}_{\varepsilon}(u)=P_{\varepsilon}(u)+Q_{\varepsilon}(u).
\end{equation*}
Clearly, $\mathcal{J}_{\varepsilon}\in C^1(H_{\varepsilon},\mathbb{R})$. To find solutions of \eqref{main-4-1} which concentrates in $\Lambda$ as $\varepsilon\rightarrow0$, we shall search critical points of $\mathcal{J}_{\varepsilon}$ such that $Q_{\varepsilon}$ is zero.

Now, we construct a set of approximate solutions of \eqref{main-4-1}. Set
\begin{equation*}
\delta_0=\frac{1}{10}{\rm dist}(\mathcal{M},\mathbb{R}^3\backslash\Lambda),\quad \beta\in(0,\delta_0).
\end{equation*}
We fix a cut-off function $\varphi\in C_0^{\infty}(\mathbb{R}^3)$ such that $0\leq\varphi\leq1$, $\varphi=1$ for $|z|\leq\beta$, $\varphi=0$ for $|z|\geq 2\beta$ and $|\nabla\varphi|\leq C/\beta$. Set $\varphi_{\varepsilon}(z)=\varphi(\varepsilon z)$, for any $W\in \mathcal{L}_{V_0}$ and any point $y\in\mathcal{M}^{\beta}=\{y\in\mathbb{R}^3\,\,|\,\,\inf\limits_{z\in\mathcal{M}}|y-z|\leq\beta\}$, we define
\begin{equation*}
W_{\varepsilon}^{y}(z)=\varphi_{\varepsilon}(z-\frac{y}{\varepsilon})W(z-\frac{y}{\varepsilon}).
\end{equation*}
Similarly, for $A\subset H_{\varepsilon}$, we use the notation
\begin{equation*}
A^{a}=\{u\in H_{\varepsilon}\,\,\Big|\,\,\inf_{v\in A}\|u-v\|_{H_{\varepsilon}}\leq a\}.
\end{equation*}

We want to find a solution near the set
\begin{equation*}
\mathcal{N}_{\varepsilon}=\{W_{\varepsilon}^{y}(z)\,\,\Big|\,\,y\in\mathcal{M}^{\beta}, \,\, W\in \mathcal{L}_{V_0}\}
\end{equation*}
for $\varepsilon>0$ sufficiently small.
\begin{lemma}\label{lem4-0-1}
$\mathcal{N}_{\varepsilon}$ is uniformly bounded in $H_{\varepsilon}$ and it is compact in $H_{\varepsilon}$ for any $\varepsilon>0$.
\end{lemma}
\begin{proof}
For any $W_{\varepsilon}^y\in\mathcal{N}_{\varepsilon}$, by H\"{o}lder's inequality, we have
\begin{align*}
\|W_{\varepsilon}^y\|_{H_{\varepsilon}}^2&=\int_{\mathbb{R}^3}|D_s(\varphi_{\varepsilon}W)|^2\,{\rm d }z+\int_{\mathbb{R}^3}V(\varepsilon z+y)\varphi_{\varepsilon}^2(z)W^2(z)\,{\rm d}z\\
&\leq2\int_{\mathbb{R}^3}\varphi_{\varepsilon}^2|D_s W|^2\,{\rm d}z+2\int_{\mathbb{R}^3}W^2|D_s\varphi_{\varepsilon}|^2\,{\rm d}z\\
&+\sup_{y\in\mathcal{M}^{\beta},z\in B_{2\beta/\varepsilon}(0)}V(\varepsilon z+y)\int_{B_{2\beta/\varepsilon}(0)}\varphi_{\varepsilon}^2(z)W^2(z)\,{\rm d}z\\
&\leq 2\int_{\mathbb{R}^3}|D_s W|^2\,{\rm d}z+C\int_{\mathbb{R}^3}W^2\,{\rm d}z+2\Big(\int_{\mathbb{R}^3}W^{2_s^{\ast}}\,{\rm d}z\Big)^{\frac{2}{2_s^{\ast}}}\Big(\int_{\mathbb{R}^3}|D_s\varphi_{\varepsilon}|^{\frac{3}{s}}\,{\rm d}z\Big)^{\frac{2s}{3}}
\end{align*}
and directly computations, we get
\begin{align*}
&\int_{\mathbb{R}^3}\Big|\int_{\mathbb{R}^3}\frac{|\varphi_{\varepsilon}(z)-\varphi_{\varepsilon}(y)|^2}{|z-y|^{3+2s}}\,{\rm d}y\Big|^{\frac{3}{2s}}\,{\rm d}z=\int_{\mathbb{R}^3}\Big|\int_{\mathbb{R}^3}\frac{|\varphi(z)-\varphi(y)|^2}{|z-y|^{3+2s}}\,{\rm d}y\Big|^{\frac{3}{2s}}\,{\rm d}z\\
&=\int_{\mathbb{R}^3\backslash B_{2\beta}(0)}\Big|\int_{\mathbb{R}^3}\frac{|\varphi(z)-\varphi(y)|^2}{|z-y|^{3+2s}}\,{\rm d}y\Big|^{\frac{3}{2s}}\,{\rm d}z+\int_{B_{2\beta}(0)}\Big|\int_{\mathbb{R}^3}\frac{|\varphi(z)-\varphi(y)|^2}{|z-y|^{3+2s}}\,{\rm d}y\Big|^{\frac{3}{2s}}\,{\rm d}z\\
&=\int_{\mathbb{R}^3\backslash B_{2\beta}(0)}\Big|\int_{B_{2\beta}(0)}\frac{|\varphi(z)-\varphi(y)|^2}{|z-y|^{3+2s}}\,{\rm d}y\Big|^{\frac{3}{2s}}\,{\rm d}z+\int_{B_{2\beta}(0)}\Big|\int_{\mathbb{R}^3}\frac{|\varphi(z)-\varphi(y)|^2}{|z-y|^{3+2s}}\,{\rm d}y\Big|^{\frac{3}{2s}}\,{\rm d}z\\
&\leq C\Big[\frac{1}{\beta^{\frac{3}{s}}}\int_{B_{3\beta}(0)}\Big|\int_{|z-y|\leq\beta}\frac{1}{|z-y|^{1+2s}}\,{\rm d}y\Big|^{\frac{3}{2s}}\,{\rm d}z+\int_{\mathbb{R}^3\backslash B_{2\beta}(0)}\Big|\int_{|z-y|>\beta,y\in B_{2\beta}(0)}\frac{|\varphi(z)-\varphi(y)|^2}{|z-y|^{3+2s}}\,{\rm d}y\Big|^{\frac{3}{2s}}\,{\rm d}z\\
&+\int_{B_{2\beta}(0)}\Big|\frac{1}{\beta^2}\int_{|z-y|\leq\beta}\frac{1}{|z-y|^{1+2s}}\,{\rm d}y+\int_{|z-y|>1}\frac{1}{|z-y|^{3+2s}}\,{\rm d}y\Big|^{\frac{3}{2s}}\,{\rm d}z\Big]\\
&\leq C\Big(1+\int_{\mathbb{R}^3\backslash B_{2\beta}(0)}\Big|\int_{|z-y|>\beta,y\in B_{2\beta}(0)}\frac{|\varphi(z)-\varphi(y)|^2}{|z-y|^{3+2s}}\,{\rm d}y\Big|^{\frac{3}{2s}}\,{\rm d}z\Big)\\
&=C\Big(1+\int_{\mathbb{R}^3\backslash B_{2\beta}(0)}\Big|\int_{|z-y|>\frac{|z|}{2},y\in B_{2\beta}(0)}\frac{1}{|z-y|^{3+2s}}\,{\rm d}y\Big|^{\frac{3}{2s}}\,{\rm d}z\\
&+\frac{1}{\beta^{\frac{3}{s}}}\int_{\mathbb{R}^3\backslash B_{2\beta}(0)}\Big|\int_{\beta<|z-y|\leq\frac{|z|}{2},y\in B_{2\beta}(0)}\frac{1}{|z-y|^{1+2s}}\,{\rm d}y\Big|^{\frac{3}{2s}}\,{\rm d}z\Big)\\
&=C\Big(1+\int_{\mathbb{R}^3\backslash B_{2\beta}(0)}\Big|\int_{|z-y|>\frac{|z|}{2},y\in B_{2\beta}(0)}\frac{1}{|z-y|^{3+2s}}\,{\rm d}y\Big|^{\frac{3}{2s}}\,{\rm d}z\\
&+\frac{1}{\beta^{\frac{3}{s}}}\int_{B_{4\beta}(0)}\Big|\int_{\beta<|z-y|\leq\frac{|z|}{2}}\frac{1}{|z-y|^{1+2s}}\,{\rm d}y\Big|^{\frac{3}{2s}}\,{\rm d}z\Big)\\
&\leq C\Big(1+\int_{\mathbb{R}^3\backslash B_{2\beta}(0)}\frac{1}{|z|^{(3+2s)\frac{3}{2s}}}\,{\rm d}z\Big)\leq C.
\end{align*}
Thus, we obtain
\begin{equation}\label{equ4-0-1}
\|W_{\varepsilon}^y\|_{H_{\varepsilon}}^2\leq C\|W\|^2
\end{equation}
for all $y\in\mathcal{M}^{\beta}$, $W\in\mathcal{L}_{V_0}$ and $\varepsilon$. From the boundedness of $\mathcal{L}_{V_0}$, we see that $\mathcal{N}_{\varepsilon}$ is uniformly bounded in $H_{\varepsilon}$.

Now let $\{W_n\}$ be a sequence in $\mathcal{N}_{\varepsilon}$, then there exists $\{U_n\}\subset\mathcal{L}_{V_0}$ and $\{x_n\}\subset\mathcal{M}^{\beta}$ satisfying $W_n(z)=\varphi_{\varepsilon}(z-\frac{x_n}{\varepsilon})U_n(z-\frac{x_n}{\varepsilon})$. The compactness of $\mathcal{L}_{V_0}$ and $\mathcal{M}^{\beta}$ imply that the existence of $U_0\in\mathcal{L}_{V_0}$ and $x_0\in\mathcal{M}^{\beta}$ such that $U_n\rightarrow U$ in $H^s(\mathbb{R}^3)$ and $x_n\rightarrow x_0$ in $\mathbb{R}^3$, up to subsequences.

Define $W_0(z)=\varphi_{\varepsilon}(z-\frac{x_0}{\varepsilon})U_0(z-\frac{x_0}{\varepsilon})$, we have $W_0\in\mathcal{N}_{\varepsilon}$. From \eqref{equ4-0-1}, it is easy to show that $W_n\rightarrow W_0$ in $H_{\varepsilon}$.
\end{proof}
For $W^{\ast}\in \mathcal{L}_{V_0}$ arbitrary but fixed, we define
\begin{equation*}
W_{\varepsilon,\tau}(z):=\varphi(\varepsilon z)W_{\tau}^{\ast}(z)=\tau^{s+t}\varphi(\varepsilon z)W^{\ast}(\tau z),
\end{equation*}
we will show that $\mathcal{J}_{\varepsilon}$ possesses the mountain-pass geometry.

Similar to the proof of Lemma \ref{lem3-1}, we can conclude that $\mathcal{J}_{\varepsilon}(u)>0$ for $\|u\|_{H_{\varepsilon}}$ small and there exists $\tau_0>0$ such that $\mathcal{I}_{V_0}(W_{\tau_0}^{\ast})<-3$, where $W_{\tau_0}^{\ast}(z)=\tau_0^{s+t}W^{\ast}(\tau_0 z)$.
\begin{lemma}\label{lem4-0-2}
\begin{equation*}
\sup_{\tau\in[0,\tau_0]}\Big|\mathcal{J}_{\varepsilon}(W_{\varepsilon,\tau})-\mathcal{I}_{V_0}(W^{\ast}_{\tau}(z))\Big|\rightarrow0\quad \text{as}\,\, \varepsilon\rightarrow0.
\end{equation*}
\end{lemma}
\begin{proof}
Since ${\rm supp}(W_{\varepsilon,\tau})\subset\Lambda_{\varepsilon}$, we have $Q_{\varepsilon}(W_{\varepsilon,\tau})\equiv0$ and so $\mathcal{J}_{\varepsilon}(W_{\varepsilon,\tau})=P_{\varepsilon}(W_{\varepsilon,\tau})$. Then for any $\tau\in[0,\tau_0]$, we get
\begin{align*}
&\Big|P_{\varepsilon}(W_{\varepsilon,\tau})-\mathcal{I}_{V_0}(W_{\tau}^{\ast}(z))\Big|\leq\frac{1}{2}\Big|\int_{\mathbb{R}^3}(|D_sW_{\varepsilon,\tau}|^2-|D_sW_{\tau}^{\ast}|^2)\,{\rm d}z\Big|+\frac{1}{2}\Big|\int_{\mathbb{R}^3}(V(\varepsilon z)W_{\varepsilon,\tau}^2-V_0(W_{\tau}^{\ast})^2)\,{\rm d}z\Big|\\
&+\frac{1}{4}\int_{\mathbb{R}^3}(\phi_{W_{\varepsilon,\tau}}^tW_{\varepsilon,\tau}^2-\phi_{W_{\tau}^{\ast}}^t(W_{\tau}^{\ast})^2)\,{\rm d }z+\Big|\int_{\mathbb{R}^3}(G(W_{\tau}^{\ast})-F(\varepsilon z,W_{\varepsilon,\tau}))\,{\rm d}z\Big|\\
&:=\frac{1}{2}I_1+\frac{1}{2}I_2+\frac{1}{4}I_3+I_4.
\end{align*}
In order to estimate $I_{i} (i=1,2,3,4)$, we set $h(\tau)=\frac{\tau^{s+t}}{1+\tau^{3+2s}|z|^{3+2s}}$ for $\tau\in[0,+\infty)$ and $|z|>0$. Directly computations, we see that $h(\tau)$ attains its maximum at $\tau_{max}=\Big(\frac{s+t}{(3+t-s)|z|^{3+2s}}\Big)^{\frac{1}{3+2s}}$ and
\begin{equation*}
\sup_{\tau\in[0,+\infty)}h(\tau)=h(\tau_{max})=\frac{(3+t-s)}{3+2s}\Big(\frac{s+t}{3+t-s}\Big)^{\frac{s+t}{3+2s}}\frac{1}{|z|^{s+t}}.
\end{equation*}
Observe that $|z|\geq\Big(\frac{s+t}{3+t-s}\Big)^{\frac{1}{3+2s}}\frac{1}{\tau_0}$, i.e., $\tau_{max}\leq\tau_0$, we have that
\begin{equation*}
\sup_{\tau\in[0,\tau_0]}h(\tau)=h(\tau_{max}).
\end{equation*}
If $|z|<\Big(\frac{s+t}{3+t-s}\Big)^{\frac{1}{3+2s}}\frac{1}{\tau_0}$, i.e., $\tau_{max}>\tau_0$, we have that
\begin{equation*}
\sup_{\tau\in[0,\tau_0]}h(\tau)=h(\tau_0).
\end{equation*}
Now, by $(ii)$ of Proposition \ref{lem3-4}, Fubini's Theorem and $W\in C^{1,\alpha}(\mathbb{R}^3)$, we have that
\begin{align*}
A_1&=\tau^{2(s+t)}\Big|\int_{\mathbb{R}^3}\int_{\mathbb{R}^3}\frac{1}{|z-y|^{3+2s}}\Big((\varphi_{\varepsilon}^2(z)-1)|W(\tau z)-W(\tau y)|^2+|\varphi_{\varepsilon}(z)-\varphi_{\varepsilon}(y)|^2W^2(\tau y)\\
&+2\varphi_{\varepsilon}(z)(\varphi_{\varepsilon}(z)-\varphi_{\varepsilon}(y))(W(\tau z)-W(\tau y))W(\tau y)\Big)\,{\rm d}y\,{\rm d}z\Big|\\
&\leq\tau^{2(s+t)}\Big(\int_{\mathbb{R}^3}\int_{\mathbb{R}^3}|\varphi_{\varepsilon}^2(z)-1|\frac{|W(\tau z)-W(\tau y)|^2}{|z-y|^{3+2s}}\,{\rm d}y\,{\rm d}z+2(1+\int_{\mathbb{R}^3}\varphi_{\varepsilon}^2|D_sW(\tau z)|^2\,{\rm d}z)\\
&\int_{\mathbb{R}^3}\int_{\mathbb{R}^3}W^2(\tau z)\frac{|\varphi_{\varepsilon}(z)-\varphi_{\varepsilon}(y)|^2}{|z-y|^{3+2s}}\,{\rm d}y\,{\rm d}z\Big)\\
&\leq\int_{\mathbb{R}^3}\int_{\mathbb{R}^3}|\varphi_{\varepsilon}^2(z)-1|(\chi_{\{|z-y|<1\}}\frac{\tau_0^{2(s+t+1)}}{|z-y|^{1+2s}}+\chi_{\{|z-y|>1\}}\frac{\tau_0^{2(s+t)}}{|z-y|^{3+2s}})\,{\rm d}y\,{\rm d}z\\
&+C(\tau_0^2+1)\int_{\mathbb{R}^3}\int_{\mathbb{R}^3}\max\{h^2(\tau_{max}),h^2(\tau_0)\}\frac{|\varphi_{\varepsilon}(z)-\varphi_{\varepsilon}(y)|^2}{|z-y|^{3+2s}}\,{\rm d}y\,{\rm d}z.
\end{align*}
Thus, the Lebesgue Dominated Convergence Theorem implies that $\sup\limits_{\tau\in[0,\tau_0]}A_1\rightarrow0$ as $\varepsilon\rightarrow0$.

For $A_2$. Since
\begin{align*}
A_2&=\tau^{2(s+t)}\Big|\int_{\mathbb{R}^3}(V(\varepsilon z)-V_0)\varphi_{\varepsilon}^2(z)W^2(\tau z)\,{\rm d}z+V_0\int_{\mathbb{R}^3}(\varphi_{\varepsilon}^2(z)-1)W^2(\tau z)\,{\rm d }z\Big|\\
&\leq\int_{\mathbb{R}^3}(V(\varepsilon z)-V_0)\varphi_{\varepsilon}^2(z)\max\{h^2(\tau_{max}),h^2(\tau_0)\}\,{\rm d}z\\
&+V_0\int_{\mathbb{R}^3}|\varphi_{\varepsilon}^2(z)-1|\max\{h^2(\tau_{max}),h^2(\tau_0)\}\,{\rm d}z,
\end{align*}
by the Lebesgue Dominated Convergence Theorem, we obtain that $\sup\limits_{\tau\in[0,\tau_0]}A_2\rightarrow0$ as $\varepsilon\rightarrow0$.

For $A_3$. Similarly arguments as above proof of $A_2$, we have that
\begin{align*}
A_3&\leq\tau^{4(s+t)}\int_{\mathbb{R}^3}\int_{\mathbb{R}^3}\frac{|\varphi_{\varepsilon}^2(z)\varphi_{\varepsilon}^2(y)-1|W^2(\tau y)W^2(\tau z)}{|z-y|^{3-2t}}\,{\rm d}y\,{\rm d}z\\
&\leq\int_{\mathbb{R}^3}\int_{\mathbb{R}^3}|\varphi_{\varepsilon}^2(z)\varphi_{\varepsilon}^2(y)-1|\frac{\max\{h^2(\tau_{max}),h^2(\tau_0)\}_y\max\{h^2(\tau_{max}),h^2(\tau_0)\}_z}{|z-y|^{3-2t}}\,{\rm d}y\,{\rm d}z.
\end{align*}
Using the Lebesgue Dominated Convergence Theorem, we get that $\sup\limits_{\tau\in[0,\tau_0]}A_3\rightarrow0$ as $\varepsilon\rightarrow0$.

For $A_4$. From $W\in L^{\infty}(\mathbb{R}^3)$ and \eqref{equ3-4}, we deduce that
\begin{align*}
A_4&\leq\int_{\mathbb{R}^3}\Big|G(\tau^{s+t}\varphi_{\varepsilon}W(\tau z))-G(\tau^{s+t}W(\tau z))\Big|\,{\rm d}z\leq C\tau^{2(s+t)}\int_{\mathbb{R}^3}(W^2(\tau z)+W^{2_s^{\ast}}(\tau z))|\varphi_{\varepsilon}(z)-1|\,{\rm d}z\\
&\leq C\int_{\mathbb{R}^3}\tau^{2(s+t)}W^2(\tau z)|\varphi_{\varepsilon}(z)-1|\,{\rm d}z\leq C\int_{\mathbb{R}^3}\max\{h^2(\tau_{max}),h^2(\tau_0)\}|\varphi_{\varepsilon}(z)-1|\,{\rm d}z.
\end{align*}
Thus, $\sup\limits_{\tau\in[0,\tau_0]}A_4\rightarrow0$ as $\varepsilon\rightarrow0$. Therefore, $\mathcal{J}_{\varepsilon}(W_{\varepsilon,\tau})\rightarrow \mathcal{I}_{V_0}(W_{\tau}^{\ast})$ as $\varepsilon\rightarrow0$, uniformly on $\tau\in[0,\tau_0]$.
\end{proof}

Since $0\in\mathcal{M}$ and $\Lambda$ is an open set, there exists $R>0$ such that $B_R(0)\subset \Lambda$, and by Proposition \ref{pro3-4} $(ii)$, we have
\begin{equation*}
\int_{\mathbb{R}^3\backslash\Lambda_{\varepsilon}}W_{\varepsilon,\tau_0}^2\,{\rm d}z\leq\tau_0^{2(s+t)}\int_{\mathbb{R}^3\backslash B_{R/\varepsilon}(0)}(W^{\ast}(\tau_0 z))^2\,{\rm d}z\leq C\frac{\varepsilon^{4s+3}}{R^{4s+3}}\leq C\varepsilon^{4s+3}
\end{equation*}
which implies that $Q_{\varepsilon}(W_{\varepsilon,\tau_0})\equiv0$ for $\varepsilon>0$ small. Thus, by Lemma \ref{lem4-0-2}, we have
\begin{align*}
&\mathcal{J}_{\varepsilon}(W_{\varepsilon,\tau_0})=P_{\varepsilon}(W_{\varepsilon,\tau_0})=\mathcal{I}_{V_0}(W_{\tau_0}^{\ast})+o(1)<-2\quad\text{for}\,\,\varepsilon>0\,\,\text{small}.
\end{align*}
Therefore, we can define the Mountain-Pass level of $\mathcal{J}_{\varepsilon}$ given by
\begin{equation*}
\mathcal{C}_{\varepsilon}:=\inf_{\gamma\in\mathcal{A}_{\varepsilon}}\max_{\tau\geq0}\mathcal{J}_{\varepsilon}(\gamma(\tau)),
\end{equation*}
where $\mathcal{A}_{\varepsilon}=\{\gamma\in C([0,1],H_{\varepsilon})\,\,|\,\, \gamma(0)=0,\,\,\gamma(1)=W_{\varepsilon,\tau_0}\}$. Furthermore, by well-known arguments (see for instance \cite{BJ,HL1} for a proof in a local setting that extends smoothly to our case) it is possible to prove the following Lemma.
\begin{lemma}\label{lem4-0-3}
\begin{equation}\label{equ4-1}
\lim_{\varepsilon\rightarrow0}\mathcal{C}_{\varepsilon}=\lim_{\varepsilon\rightarrow0}\mathcal{D}_{\varepsilon}:=\lim_{\varepsilon\rightarrow0}\max_{\tau\in[0,1]}\mathcal{J}_{\varepsilon}(\gamma_{\varepsilon}(\tau))=c_{V_0}
\end{equation}
where $\gamma_{\varepsilon}(\tau)=W_{\varepsilon,\tau\tau_0}$ for $\tau\in[0,1]$ and $c_{V_0}=\mathcal{I}_{V_0}(W^{\ast})$ for $W^{\ast}\in\mathcal{L}_{V_0}$.
\end{lemma}
\begin{proof}
First we will prove that $\limsup\limits_{\varepsilon\rightarrow0}\mathcal{C}_{\varepsilon}\leq c_{V_0}$. Setting $\gamma_{\varepsilon}(\tau)=W_{\varepsilon,\tau\tau_0}$ for $\tau\in[0,1]$, we get $\gamma_{\varepsilon}\in\Gamma_{\varepsilon}$ and from Lemma \ref{lem4-0-2}, we have
\begin{align*}
\limsup_{\varepsilon\rightarrow0}\mathcal{C}_{\varepsilon}&\leq\limsup_{\varepsilon\rightarrow0}\max_{\tau\in[0,1]}\mathcal{J}_{\varepsilon}(\gamma_{\varepsilon}(\tau))
\leq\limsup_{\varepsilon\rightarrow0}\max_{\tau\in[0,\tau_0]}\mathcal{J}_{\varepsilon}(W_{\varepsilon,\tau})\leq\max_{\tau\in[0,\tau_0]}\mathcal{I}_{V_0}(W_{\tau}^{\ast})\\
&\leq\max_{\tau\in[0,+\infty)}\mathcal{I}_{V_0}(W_{\tau}^{\ast})=\mathcal{I}_{V_0}(W)=c_{V_0}.
\end{align*}
which we conclude the first part of the proof. Next we shall prove that $\liminf\limits_{\varepsilon\rightarrow0}\mathcal{C}_{\varepsilon}\geq c_{V_0}$. Assume the contrary that $\liminf\limits_{\varepsilon\rightarrow0}\mathcal{C}_{\varepsilon}< c_{V_0}$. Then there exist $\delta_0>0$, $\varepsilon_n\rightarrow0$ and $\gamma_n:=\gamma_{\varepsilon_n}\in\mathcal{A}_{\varepsilon_n}$ satisfying $\mathcal{J}_{\varepsilon}(\gamma_n(\tau))<c_{V_0}-\delta_0$ for $\tau\in[0,1]$.
Since $P_{\varepsilon_n}(\gamma_n(0))=0$ and $P_{\varepsilon_n}(\gamma_n(1))\leq\mathcal{J}_{\varepsilon_n}(\gamma_n(1))=\mathcal{J}_{\varepsilon_n}(W_{\varepsilon_n,\tau_0})<-2$, we can find $\tau_n\in(0,1)$ such that $P_{\varepsilon_n}(\gamma_n(\tau))\geq-1$ for $\tau\in[0,\tau_n]$ and $P_{\varepsilon_n}(\gamma_n(\tau_n))=-1$.
Since
\begin{align*}
P_{\varepsilon_n}(\gamma_n(\tau))&=\mathcal{I}_{V_0}(\gamma_n(\tau))+\frac{1}{2}\int_{\mathbb{R}^3}(V(\varepsilon_n z)-V_0)\gamma_n^2(\tau)\,{\rm d}z+\int_{\mathbb{R}^3}[G(\gamma_n(\tau))-F(\varepsilon_n z,\gamma_n(\tau)))\,{\rm d}z\\
&\geq\mathcal{I}_{V_0}(\gamma_n(\tau))+\frac{1}{2}\int_{\mathbb{R}^3}(V(\varepsilon_n z)-V_0)\gamma_n^2(\tau)\,{\rm d}z\geq\mathcal{I}_{V_0}(\gamma_n(\tau)),\quad \forall\tau\in[0,\tau_n],
\end{align*}
then
\begin{align*}
\mathcal{I}_{V_0}(\gamma_n(\tau_n))\leq P_{\varepsilon_n}(\gamma_n(\tau_n))=-1<0.
\end{align*}
Recalling that the mountain pass level for $\mathcal{I}_{V_0}$ corresponds to the least energy level, we have $\max\limits_{\tau\in[0,\tau_n]}\mathcal{I}_{V_0}(\gamma_n(\tau))\geq c_{V_0}$. Since $Q_{\varepsilon_n}(\gamma_n(\tau))\geq0$, by the estimates above we obtain
\begin{align*}
c_{V_0}-\delta_0&>\max_{\tau\in[0,1]}\mathcal{J}_{\varepsilon_n}(\gamma_n(\tau))\geq\max_{\tau\in[0,1]}P_{\varepsilon_n}(\gamma_n(\tau))\geq\max_{\tau\in[0,\tau_n]}P_{\varepsilon_n}(\gamma_n(\tau))\\
&\geq\max_{\tau\in[0,\tau_n]}\mathcal{I}_{V_0}(\gamma_n(\tau))\geq c_{V_0}.
\end{align*}
This contradiction completes the proof.
\end{proof}

\begin{lemma}\label{lem4-1}
There exists a small $d_0>0$ such that for any $\{\varepsilon_i\}$, $\{u_{\varepsilon_i}\}$ satisfying $\lim\limits_{i\rightarrow\infty}\varepsilon_i\rightarrow0$, $u_{\varepsilon_i}\in \mathcal{N}_{\varepsilon_i}^{d_0}$ and
\begin{equation*}
\lim_{i\rightarrow\infty}\mathcal{J}_{\varepsilon_i}(u_{\varepsilon_i})\leq c_{V_0}\quad \text{and}\quad \lim_{i\rightarrow\infty}\mathcal{J}_{\varepsilon_i}'(u_{\varepsilon_i})=0,
\end{equation*}
there exist, up to a subsequence, $\{x_i\}\subset\mathbb{R}^3$, $x_0\in\mathcal{M}$, $W\in \mathcal{L}_{V_0}$ such that
\begin{equation*}
\lim_{i\rightarrow\infty}|\varepsilon_ix_i-x_0|=0\quad \text{and}\quad \lim_{i\rightarrow\infty}\|u_{\varepsilon_i}-\varphi_{\varepsilon}(\cdot-x_{i})W(\cdot-x_{i})\|_{H_{\varepsilon_i}}=0.
\end{equation*}
\end{lemma}

\begin{proof}
In the proof we will drop the index $i$ and write $\varepsilon$ instead of $\varepsilon_i$ for simplicity, and we still use $\varepsilon$ after taking a subsequence. By the definition of $\mathcal{N}_{\varepsilon}^{d_0}$, there exist $\{W_{\varepsilon}\}\subset\mathcal{L}_{V_0}$ and $\{x_{\varepsilon}\}\subset\mathcal{M}^{\beta}$ such that for $\varepsilon$ small,
\begin{equation*}
\|u_{\varepsilon}-\varphi_{\varepsilon}(\cdot-\frac{x_{\varepsilon}}{\varepsilon})W_{\varepsilon}(\cdot-\frac{x_{\varepsilon}}{\varepsilon})\|_{H_{\varepsilon}}\leq\frac{3}{2}d_0.
\end{equation*}
Since $\mathcal{L}_{V_0}$ and $\mathcal{M}^{\beta}$ are compact, there exist $W_0\in\mathcal{L}_{V_0}$, $x_0\in\mathcal{M}^{\beta}$ such that $W_{\varepsilon}\rightarrow W_0$ in $H^s(\mathbb{R}^3)$ and $x_{\varepsilon}\rightarrow x_0$ as $\varepsilon\rightarrow0$. Thus, for $\varepsilon>0$ small,
\begin{equation}\label{equ4-2}
\|u_{\varepsilon}-\varphi_{\varepsilon}(\cdot-\frac{x_{\varepsilon}}{\varepsilon})W_0(\cdot-\frac{x_{\varepsilon}}{\varepsilon})\|_{H_{\varepsilon}}\leq2d_0.
\end{equation}

{\bf Step 1.} We claim that
\begin{equation}\label{equ4-3}
\lim_{\varepsilon\rightarrow0}\sup_{y\in A_{\varepsilon}}\int_{B_1(y)}|u_{\varepsilon}|^2\,{\rm d}z=0,
\end{equation}
where $A_{\varepsilon}=B_{3\beta/\varepsilon}(x_{\varepsilon}/\varepsilon)\backslash B_{\beta/2\varepsilon}(x_{\varepsilon}/\varepsilon)$. Suppose by contradiction that
\begin{equation*}
\liminf_{\varepsilon\rightarrow0}\sup_{y\in A_{\varepsilon}}\int_{B_1(y)}|u_{\varepsilon}|^2\,{\rm d}z>0.
\end{equation*}
Thus, there exists $y_{\varepsilon}\in A_{\varepsilon}$ such that $\int_{B_1(y_{\varepsilon})}|u_{\varepsilon}|^2\,{\rm d}z>0$ for $\varepsilon>0$ small. Since $y_{\varepsilon}\in A_{\varepsilon}$, there exists $y^{\ast}\in\mathcal{M}^{4\beta}\subset\Lambda$ such that $\varepsilon y_{\varepsilon}\rightarrow y^{\ast}$ as $\varepsilon\rightarrow0$. Set $v_{\varepsilon}(z)=u_{\varepsilon}(z+y_{\varepsilon})$, then for $\varepsilon>0$ small,
\begin{equation}\label{equ4-4}
\int_{B_1(0)}|v_{\varepsilon}|^2\,{\rm d}z>0.
\end{equation}
Thus, up to a subsequence, we may assume that there exists $v\in H^s(\mathbb{R}^3)$ such that $v_{\varepsilon}\rightharpoonup v$ in $H^s(\mathbb{R}^3)$, $v_{\varepsilon}\rightarrow v$ in $L_{loc}^p(\mathbb{R}^3)$ for $1\leq p<2_s^{\ast}$ and $v_{\varepsilon}\rightarrow v$ a.e. in $\mathbb{R}^3$. By \eqref{equ4-4}, we see that $v\neq0$ and $v$ satisfies
\begin{equation}\label{equ4-4-0}
(-\Delta)^s v+V(y^{\ast}) v+\phi_{v}^t v=g(v)   \quad z\in \mathbb{R}^3.
\end{equation}
Indeed, by the definition of weakly convergence, we have
\begin{align*}
\frac{C_s}{2}\int_{\mathbb{R}^3}\int_{\mathbb{R}^3}\frac{(v_{\varepsilon}(z)-v_{\varepsilon}(y))(\varphi(z)-\varphi(y)}{|z-y|^{3+2s}}\,{\rm d }y\,{\rm d }z+\int_{\mathbb{R}^3}V(y^{\ast})v_{\varepsilon}\varphi\,{\rm d}z\rightarrow\\
\frac{C_s}{2}\int_{\mathbb{R}^3}\int_{\mathbb{R}^3}\frac{(v(z)-v(y))(\varphi(z)-\varphi(y)}{|z-y|^{3+2s}}\,{\rm d }y\,{\rm d }z+\int_{\mathbb{R}^3}V(y^{\ast})v\varphi\,{\rm d}z
\end{align*}
for any $\varphi\in C_0^{\infty}(\mathbb{R}^3)$. Now given $\varphi\in C_0^{\infty}(\mathbb{R}^3)$, we have $\|\varphi(\cdot-y_{\varepsilon})\|_{H_{\varepsilon}}\leq C$ and so $\langle\mathcal{J}_{\varepsilon}'(u_{\varepsilon}),\varphi(\cdot-y_{\varepsilon})\rangle\rightarrow0$ as $\varepsilon\rightarrow0$. Using the fact that $v_{\varepsilon}\rightarrow v$ in $L_{loc}^p(\mathbb{R}^3)$ for $1\leq p<2_s^{\ast}$, the Lebesgue dominated convergence Theorem, the boundedness of ${\rm supp}(\varphi)$ and $(g_0)$--$(g_1)$, it follows that
\begin{align*}
\int_{\mathbb{R}^3}(V(\varepsilon z+\varepsilon y_{\varepsilon})-V(y^{\ast}))v_{\varepsilon}\varphi\,{\rm d }z\rightarrow0,\quad \int_{\mathbb{R}^3}(\phi_{v_{\varepsilon}}^tv_{\varepsilon}-\phi_v^tv)\varphi\,{\rm d }z\rightarrow0,
\end{align*}
\begin{align*}
\int_{\mathbb{R}^3\backslash\Lambda_{\varepsilon}}u_{\varepsilon}(z)\varphi(z-y_{\varepsilon})\,{\rm d}z=\int_{\mathbb{R}^3\backslash\Lambda_{\varepsilon}+y_{\varepsilon}}v_{\varepsilon}(z)\varphi(z)\,{\rm d}z\rightarrow0
\end{align*}
and
\begin{align*}
\int_{\mathbb{R}^3}(f(\varepsilon z+\varepsilon y_{\varepsilon},v_{\varepsilon})-g(v))\varphi\,{\rm d }z\rightarrow0
\end{align*}
for any $\varphi\in C_0^{\infty}(\mathbb{R}^3)$. Therefore, we get that
\begin{align*}
\frac{C_s}{2}\int_{\mathbb{R}^3}\int_{\mathbb{R}^3}\frac{(v(z)-v(y))(\varphi(z)-\varphi(y))}{|z-y|^{3+2s}}\,{\rm d}y\,{\rm d}z&+\int_{\mathbb{R}^3}V(y^{\ast})v\varphi\,{\rm d}z+\int_{\mathbb{R}^3}\phi_v^tv\varphi\,{\rm d}z-\int_{\mathbb{R}^3}g(v)\varphi\,{\rm d}z=0
\end{align*}
for any $\varphi\in C_0^{\infty}(\mathbb{R}^3)$. Since $\varphi$ is arbitrary and $C_0^{\infty}(\mathbb{R}^3)$ is dense in $H_{\varepsilon}$, it follows that $v$ satisfies \eqref{equ4-4-0}.

Thus, we have
\begin{align*}
&c_{V(y^{\ast})}\leq \mathcal{I}_{V(y^{\ast})}(v)=\mathcal{I}_{V(y^{\ast})}(v)-\frac{1}{4s+2t-3}\mathcal{G}_{V(y^{\ast})}(v)\\
&=s\int_{\mathbb{R}^3}V(y^{\ast})|v|^2\,{\rm d}z+\frac{s+t}{4s+2t-3}\int_{\mathbb{R}^3}[g(v)v-\frac{4s+2t}{s+t}G(v)]\,{\rm d}z\\
&\leq s\|V\|_{L^{\infty}(\overline{\Lambda})}\int_{\mathbb{R}^3}|v|^2\,{\rm d}z+\frac{s+t}{4s+2t-3}\int_{\mathbb{R}^3}[g(v)v-\frac{4s+2t}{s+t}G(v)]\,{\rm d}z.
\end{align*}
Hence, for sufficiently large $r>0$, by Fatou's Lemma, we have that
\begin{align*}
&\liminf_{\varepsilon\rightarrow0}\Big[s\|V\|_{L^{\infty}(\overline{\Lambda})}\int_{B_R(y_{\varepsilon})}|u_{\varepsilon}|^2\,{\rm d}z+\frac{s+t}{4s+2t-3}\int_{B_r(y_{\varepsilon})}[g(u_{\varepsilon})u_{\varepsilon}-\frac{4s+2t}{s+t}G(u_{\varepsilon})]\,{\rm d}z\\
&=\liminf_{\varepsilon\rightarrow0}\Big[s\|V\|_{L^{\infty}(\overline{\Lambda})}\int_{B_r(0)}|v_{\varepsilon}|^2\,{\rm d}z+\frac{s+t}{4s+2t-3}\int_{B_r(0)}[g(v_{\varepsilon})v_{\varepsilon}-\frac{4s+2t}{s+t}G(v_{\varepsilon})]\,{\rm d}z\\
&\geq\Big[s\|V\|_{L^{\infty}(\overline{\Lambda})}\int_{B_r(0)}|v|^2\,{\rm d}z+\frac{s+t}{4s+2t-3}\int_{B_r(0)}[g(v)v-\frac{4s+2t}{s+t}G(v)]\,{\rm d}z\\
&\geq\frac{1}{2}\Big[s\|V\|_{L^{\infty}(\overline{\Lambda})}\int_{\mathbb{R}^3}|v|^2\,{\rm d}z+\frac{s+t}{4s+2t-3}\int_{\mathbb{R}^3}[g(v)v-\frac{4s+2t}{s+t}G(v)]\,{\rm d}z\Big]\\
&\geq\frac{1}{2}c_{V(x^{\ast})}>0.
\end{align*}
On the other hand, by the Sobolev embedding theorem, \eqref{equ3-4} and \eqref{equ4-2}, one has
\begin{align*}
&s\|V\|_{L^{\infty}(\overline{\Lambda})}\int_{B_r(y_{\varepsilon})}|u_{\varepsilon}|^2\,{\rm d}z+\frac{s+t}{4s+2t-3}\int_{B_r(y_{\varepsilon})}[g(u_{\varepsilon})u_{\varepsilon}-\frac{4s+2t}{s+t}G(u_{\varepsilon})]\,{\rm d}z\\
&\leq Cd_0+C\int_{B_r(y_{\varepsilon})}\Big|\varphi(\varepsilon z-x_{\varepsilon})W_0(z-\frac{x_{\varepsilon}}{\varepsilon})\Big|^2\,{\rm d}z\leq Cd_0+C\int_{B_r(y_{\varepsilon}-\frac{x_{\varepsilon}}{\varepsilon})}|W_0|^2\,{\rm d}z
\end{align*}
Observing that $y_{\varepsilon}\in A_{\varepsilon}$, implies that $|y_{\varepsilon}-\frac{x_{\varepsilon}}{\varepsilon}|\geq\frac{\beta}{2\varepsilon}$, then for $\varepsilon>0$ small enough, there hold
\begin{equation*}
\int_{B_r(y_{\varepsilon}-\frac{x_{\varepsilon}}{\varepsilon})}|W_0|^2\,{\rm d}z=o(1),
\end{equation*}
where $o(1)\rightarrow0$ as $\varepsilon\rightarrow0$. Thus, we have proved that
\begin{align*}
\frac{1}{2}c_{V(y^{\ast})}&\leq s\|V\|_{L^{\infty}(\overline{\Lambda})}\int_{B_r(y_{\varepsilon})}|u_{\varepsilon}|^2\,{\rm d}z+\frac{s+t}{4s+2t-3}\int_{B_r(y_{\varepsilon})}[g(u_{\varepsilon})u_{\varepsilon}-\frac{4s+2t}{s+t}G(u_{\varepsilon})]\,{\rm d}z\\
&\leq Cd_0+o(1).
\end{align*}
This leads to a contradiction if $d_0$ is small enough.

From \eqref{equ4-3} and the Vanishing Lemma \ref{lem2-3}, we conclude that
\begin{equation}\label{equ4-8}
\lim_{\varepsilon\rightarrow0}\int_{A_{\varepsilon}^1}|u_{\varepsilon}|^p\,{\rm d}z=0\quad p\in(2,2_s^{\ast}),
\end{equation}
where $A_{\varepsilon}^1=B_{2\beta/\varepsilon}(\frac{x_{\varepsilon}}{\varepsilon})\backslash B_{\beta/\varepsilon}(\frac{x_{\varepsilon}}{\varepsilon})$. Indeed, taking a smooth cut-off function $\psi_{\varepsilon}\in C_0^{\infty}(\mathbb{R}^3)$ such that $\psi_{\varepsilon}=1$ on $B_{2\beta/\varepsilon}(\frac{x_{\varepsilon}}{\varepsilon})\backslash B_{\beta/\varepsilon}(\frac{x_{\varepsilon}}{\varepsilon})$, $\psi_{\varepsilon}=0$ on $A_{\varepsilon}^2=B_{3\beta/\varepsilon-1}(\frac{x_{\varepsilon}}{\varepsilon})\backslash B_{\beta/2\varepsilon+1}(\frac{x_{\varepsilon}}{\varepsilon})$.  Since $u_{\varepsilon}\in H_{\varepsilon}$ and using $(V_0)$, it is easy to check that $u_{\varepsilon}\psi_{\varepsilon}\in H^s(\mathbb{R}^3)$. Moreover,
\begin{equation*}
\sup_{y\in A_{\varepsilon}}\int_{B_1(y)}|u_{\varepsilon}|^2\,{\rm d}z\geq\sup_{y\in\mathbb{R}^3}\int_{B_1(y)}|u_{\varepsilon}\psi_{\varepsilon}|^2\,{\rm d}z.
\end{equation*}
By Vanishing Lemma \ref{lem2-3}, we have that for $p\in(2,2_s^{\ast})$,
\begin{equation*}
\int_{\mathbb{R}^3}|u_{\varepsilon}\psi_{\varepsilon}|^p\,{\rm d}z\rightarrow0\quad \text{as}\,\, \varepsilon\rightarrow0.
\end{equation*}
Since $A_{\varepsilon}^1\subset A_{\varepsilon}^2$ for $\varepsilon>0$ small, so \eqref{equ4-8} holds.

{\bf Step 2.} Set $u_{\varepsilon,1}(z)=\varphi(\varepsilon z-x_{\varepsilon})u_{\varepsilon}(z)$, $u_{\varepsilon,2}(z)=(1-\varphi(\varepsilon z-x_{\varepsilon}))u_{\varepsilon}(z)$. Direct computation, we have
\begin{align}\label{equ4-9-0}
\int_{\mathbb{R}^3}|D_s u_{\varepsilon}|^2\,{\rm d}z&=\int_{\mathbb{R}^3}|D_s u_{\varepsilon,1}|^2\,{\rm d}z+\int_{\mathbb{R}^3}|D_s u_{\varepsilon,2}|^2\,{\rm d}z\nonumber\\
&+2\int_{\mathbb{R}^3}\int_{\mathbb{R}^3}\frac{(u_{\varepsilon,1}(x)-u_{\varepsilon,1}(y))(u_{\varepsilon,2}(x)-u_{\varepsilon,2}(y))}{|x-y|^{3+2s}}\,{\rm d}y\,{\rm d}z\nonumber\\
&\geq\int_{\mathbb{R}^3}|D_su_{\varepsilon,1}|^2\,{\rm d}z+\int_{\mathbb{R}^3}|D_s u_{\varepsilon,2}|^2\,{\rm d}z+o(1).
\end{align}
Indeed,
\begin{align*}
&(u_{\varepsilon,1}(z)-u_{\varepsilon,1}(y))(u_{\varepsilon,2}(z)-u_{\varepsilon,2}(y))\\
&=\varphi(\varepsilon z-x_{\varepsilon})(1-\varphi(\varepsilon z-x_{\varepsilon}))|D_su_{\varepsilon}|^2+\varphi(\varepsilon z-x_{\varepsilon})(\varphi(\varepsilon z-x_{\varepsilon})-\varphi(\varepsilon y-x_{\varepsilon}))\\
&(u_{\varepsilon}(z)-u_{\varepsilon}(y))u_{\varepsilon}(y)+(1-\varphi(\varepsilon z-x_{\varepsilon}))(\varphi(\varepsilon z-x_{\varepsilon})-\varphi(\varepsilon y-x_{\varepsilon}))(u_{\varepsilon}(z)-u_{\varepsilon}(y))u_{\varepsilon}(y)\\
&-(\varphi(\varepsilon z-x_{\varepsilon})-\varphi(\varepsilon y-x_{\varepsilon}))|u_{\varepsilon}(y)|^2\\
&:=\varphi(\varepsilon z-x_{\varepsilon})(1-\varphi(\varepsilon z-x_{\varepsilon}))|D_su_{\varepsilon}|^2+B_1+B_2-B_3.
\end{align*}
Next we show that $\lim\limits_{\varepsilon\rightarrow0}\int_{\mathbb{R}^3}B_i\,{\rm d}z=0$, $i=1,2,3$. If these are proved, we get
\begin{equation}\label{equ4-9-1}
\int_{\mathbb{R}^3}\int_{\mathbb{R}^3}\frac{(u_{\varepsilon,1}(x)-u_{\varepsilon,1}(y))(u_{\varepsilon,2}(x)-u_{\varepsilon,2}(y))}{|x-y|^{3+2s}}\,{\rm d}y\,{\rm d}z\geq\int_{\mathbb{R}^3}\varphi(\varepsilon z-x_{\varepsilon})(1-\varphi(\varepsilon z-x_{\varepsilon}))|D_su_{\varepsilon}|^2+o(1)
\end{equation}
and so \eqref{equ4-9-0} follows. Here $o(1)\rightarrow0$ as $\varepsilon\rightarrow0$.

Observe that
\begin{align*}
\int_{\mathbb{R}^3}B_1\,{\rm d }z&\leq\Big(\int_{\mathbb{R}^3}(\varphi(\varepsilon z-x_{\varepsilon}))^2|D_su_{\varepsilon}|^2\,{\rm d}z\Big)^{\frac{1}{2}}\Big(\int_{\mathbb{R}^3}|D_s\varphi(\varepsilon\cdot-x_{\varepsilon})|^2u_{\varepsilon}^2\,{\rm d}z\Big)^{\frac{1}{2}}\\
&\leq C\Big(\int_{\mathbb{R}^3}|D_s\varphi(\varepsilon\cdot-x_{\varepsilon})|^2u_{\varepsilon}^2\,{\rm d}z\Big)^{\frac{1}{2}}=C(\int_{\mathbb{R}^3}B_3\,{\rm d}z)^{\frac{1}{2}}
\end{align*}
and similarly, we have
\begin{equation*}
\int_{\mathbb{R}^3}B_2\,{\rm d}z\leq\Big(\int_{\mathbb{R}^3}|D_s\varphi(\varepsilon\cdot-x_{\varepsilon})|^2u_{\varepsilon}^2\,{\rm d}z\Big)^{\frac{1}{2}}=C(\int_{\mathbb{R}^3}B_3\,{\rm d}z)^{\frac{1}{2}}.
\end{equation*}
Hence, it is sufficient to prove that
\begin{equation}\label{equ4-9-2}
\lim_{\varepsilon\rightarrow0}\int_{\mathbb{R}^3}B_3\,{\rm d}z=0.
\end{equation}
In fact, direct computations, we deduce that
\begin{align*}
\int_{\mathbb{R}^3}B_3\,{\rm d}z&=\int_{\mathbb{R}^3}u_{\varepsilon}^2\int_{\mathbb{R}^3}\frac{|\varphi(\varepsilon z-x_{\varepsilon})-\varphi(\varepsilon y-x_{\varepsilon})|^2}{|z-y|^{3+2s}}\,{\rm d}y\,{\rm d}z\\
&=\varepsilon^{2s-3}\int_{\mathbb{R}^3}u_{\varepsilon}^2(\frac{z+x_{\varepsilon}}{\varepsilon})\int_{\mathbb{R}^3}\frac{|\varphi(z)-\varphi(y)|^2}{|z-y|^{3+2s}}\,{\rm d}y\,{\rm d}z\\
&\leq\varepsilon^{2s-3}\int_{\mathbb{R}^3}u_{\varepsilon}^2(\frac{z+x_{\varepsilon}}{\varepsilon})\frac{C}{\beta^2}\int_{|z-y|\leq\beta}\frac{1}{|z-y|^{1+2s}}\,{\rm d}y+\int_{|z-y|>\beta}\frac{1}{|z-y|^{3+2s}}\,{\rm d}y\,{\rm d}z\\
&\leq\frac{C\varepsilon^{2s-3}}{\beta^{2s}}\int_{\mathbb{R}^3}u_{\varepsilon}^2(\frac{z+x_{\varepsilon}}{\varepsilon})\,{\rm d}z=\frac{C\varepsilon^{2s}}{\beta^{2s}}\int_{\mathbb{R}^3}u_{\varepsilon}^2\,{\rm d}z\leq\frac{C}{\beta^{2s}}\varepsilon^{2s}.
\end{align*}
From the estimate above, we conclude that \eqref{equ4-9-2} follows. Thus \eqref{equ4-9-0} holds.

By \eqref{equ4-8}, we deduce that
\begin{align*}
\int_{\mathbb{R}^3}V(\varepsilon z)|u_{\varepsilon}|^2\,{\rm d}z\geq\int_{\mathbb{R}^3}V(\varepsilon z)|u_{\varepsilon,1}|^2\,{\rm d}z+\int_{\mathbb{R}^3}V(\varepsilon z)|u_{\varepsilon,2}|^2\,{\rm d}z
\end{align*}
\begin{align*}
\int_{\mathbb{R}^3}\phi_{u_{\varepsilon}}^t|u_{\varepsilon}|^2\,{\rm d}z\geq\int_{\mathbb{R}^3}\phi_{u_{\varepsilon,1}}^t|u_{\varepsilon,1}|^2\,{\rm d}z+\int_{\mathbb{R}^3}\phi_{u_{\varepsilon,2}}^t|u_{\varepsilon,2}|^2\,{\rm d}z
\end{align*}
\begin{align*}
\int_{\mathbb{R}^3}F(\varepsilon z,u_{\varepsilon})\,{\rm d}z=\int_{\mathbb{R}^3}F(\varepsilon z,u_{\varepsilon,1})\,{\rm d}z+\int_{\mathbb{R}^3}F(\varepsilon z,u_{\varepsilon,2})\,{\rm d}z+o(1)\,\, \text{as}\,\,\varepsilon\rightarrow0
\end{align*}
and
\begin{equation*}
Q_{\varepsilon}(u_{\varepsilon,1})=0, \quad Q_{\varepsilon}(u_{\varepsilon,2})=Q_{\varepsilon}(u_{\varepsilon})\geq0.
\end{equation*}
Hence, we get
\begin{equation}\label{equ4-10}
\mathcal{J}_{\varepsilon}(u_{\varepsilon})\geq P_{\varepsilon}(u_{\varepsilon,1})+P_{\varepsilon}(u_{\varepsilon,2})+o(1),
\end{equation}
where $o(1)\rightarrow0$ as $\varepsilon\rightarrow0$.

We now estimate $P_{\varepsilon}(u_{\varepsilon,2})$. It follows from \eqref{equ4-2} that
\begin{align*}
\|u_{\varepsilon,2}\|_{H_{\varepsilon}}&\leq6d_0+o(1),
\end{align*}
where $o(1)\rightarrow0$ as $\varepsilon\rightarrow0$ and the above inequality implies that
\begin{equation}\label{equ4-11}
\limsup\limits_{\varepsilon\rightarrow0}\|u_{\varepsilon,2}\|_{H_{\varepsilon}}\leq 6d_0.
\end{equation}
Then, by \eqref{equ3-4}, we get
\begin{align}\label{equ4-12}
P_{\varepsilon}(u_{\varepsilon,2})&\geq\frac{1}{2}\|u_{\varepsilon,2}\|_{H_{\varepsilon}}^2-\int_{\mathbb{R}^3}F(\varepsilon z,u_{\varepsilon,2})\,{\rm d}z\geq\frac{1}{4}\|u_{\varepsilon,2}\|_{H_{\varepsilon}}^2-C\|u_{\varepsilon,2}\|_{H_{\varepsilon}}^{2_s^{\ast}}\nonumber\\
&=\|u_{\varepsilon,2}\|_{H_{\varepsilon}}^2(\frac{1}{4}-C\|u_{\varepsilon,2}\|_{H_{\varepsilon}}^{2_s^{\ast}-2})\geq\|u_{\varepsilon,2}\|_{H_{\varepsilon}}^2(\frac{1}{4}-C(6d_0)^{2_s^{\ast}-2}).
\end{align}
In particular, taking $d_0>0$ small enough, we can assume that $P_{\varepsilon}(u_{\varepsilon,2})\geq0$. Hence, from \eqref{equ4-10}, it holds
\begin{equation}\label{equ4-13}
\mathcal{J}_{\varepsilon}(u_{\varepsilon})\geq P_{\varepsilon}(u_{\varepsilon,1})+o(1).
\end{equation}
Furthermore, by \eqref{equ4-8} and \eqref{equ4-9-1}, it is easy to check that
\begin{align*}
\int_{\mathbb{R}^3}\phi_{u_{\varepsilon}}^tu_{\varepsilon,1}u_{\varepsilon,2}\,{\rm d}z\leq\int_{A_{\varepsilon}^1}\phi_{u_{\varepsilon}}^t|u_{\varepsilon}|^2\,{\rm d}z\leq\|\phi_{u_{\varepsilon}}^t\|_{2_t^{\ast}}\|u_{\varepsilon}\|_{L^{\frac{12}{3+2t}}(A_{\varepsilon}^1)}^2\rightarrow0
\end{align*}
and
\begin{align*}
\int_{\mathbb{R}^3}\int_{\mathbb{R}^3}\frac{(u_{\varepsilon,1}(z)-u_{\varepsilon,1}(y))(u_{\varepsilon,2}(z)-u_{\varepsilon,2}(y))}{|z-y|^{3+2s}}\,{\rm d}y\,{\rm d}z\geq o(1).
\end{align*}
Hence, using the facts that $\langle\mathcal{J}_{\varepsilon}'(u_{\varepsilon}),u_{\varepsilon,2}\rangle\rightarrow0$ as $\varepsilon\rightarrow0$, $\langle Q_{\varepsilon}'(u_{\varepsilon}),u_{\varepsilon,2}\rangle\geq0$ and \eqref{equ3-4}, we have that
\begin{align*}
&\|u_{\varepsilon,2}\|_{H_{\varepsilon}}^2+o(1)\\
&\leq\|u_{\varepsilon,2}\|_{H_{\varepsilon}}^2+\int_{\mathbb{R}^3}\int_{\mathbb{R}^3}\frac{(u_{\varepsilon,1}(z)-u_{\varepsilon,1}(y))(u_{\varepsilon,2}(z)-u_{\varepsilon,2}(y))}{|z-y|^{3+2s}}\,{\rm d}y\,{\rm d}z+\int_{\mathbb{R}^3}V(\varepsilon z)u_{\varepsilon,1}u_{\varepsilon,2}\,{\rm d}z\\
&+\int_{\mathbb{R}^3}\phi_{u_{\varepsilon}}^tu_{\varepsilon,1}u_{\varepsilon,2}\,{\rm d}z\\
&\leq\int_{\mathbb{R}^3}\int_{\mathbb{R}^3}\frac{(u_{\varepsilon}(z)-u_{\varepsilon}(y))(u_{\varepsilon,2}(z)-u_{\varepsilon,2}(y))}{|z-y|^{3+2s}}\,{\rm d}y\,{\rm d}z+\int_{\mathbb{R}^3}V(\varepsilon x)u_{\varepsilon}u_{\varepsilon,2}\,{\rm d}z+\langle Q_{\varepsilon}(u_{\varepsilon}),u_{\varepsilon,2}\rangle\\
&+\int_{\mathbb{R}^3}\phi_{u_{\varepsilon}}^tu_{\varepsilon}u_{\varepsilon,2}\,{\rm d}z+o(1)=\int_{\mathbb{R}^3}f(\varepsilon z,u_{\varepsilon})u_{\varepsilon,2}\,{\rm d}z+o(1)\\
&\leq \eta\int_{\mathbb{R}^3}|u_{\varepsilon}u_{\varepsilon,2}|\,{\rm d}z+C\int_{\mathbb{R}^3}|u_{\varepsilon}|^{2_s^{\ast}-1}|u_{\varepsilon,2}|\,{\rm d}z+o(1)\\
&\leq\eta\|u_{\varepsilon,2}\|_{L^2}^2+C\int_{\mathbb{R}^3}\Big(|u_{\varepsilon,2}|^{2_s^{\ast}}+|u_{\varepsilon,1}|^{2_s^{\ast}-1}|u_{\varepsilon,2}|\Big)\,{\rm d}x+o(1)\leq\eta\|u_{\varepsilon,2}\|_{H_{\varepsilon}}^2+C\|u_{\varepsilon,2}\|_{H_{\varepsilon}}^{2_s^{\ast}}+o(1).
\end{align*}
Combining with \eqref{equ4-11}, we get that
\begin{equation*}
(\frac{1}{2}-Cd_0^{2_s^{\ast}-2})\|u_{\varepsilon,2}\|_{H_{\varepsilon}}^2\leq(\frac{1}{2}-C\|u_{\varepsilon,2}\|_{H_{\varepsilon}}^{2_s^{\ast}-2})\|u_{\varepsilon,2}\|_{H_{\varepsilon}}^2+o(1)\leq o(1).
\end{equation*}
Thus, taking $d_0>0$ sufficiently small, we have
\begin{equation}\label{equ4-13-0}
\lim_{\varepsilon\rightarrow0}\|u_{\varepsilon,2}\|_{H_{\varepsilon}}=0.
\end{equation}

We next estimate $P_{\varepsilon}(u_{\varepsilon,1})$. Denote $\widehat{u}_{\varepsilon}(z)=u_{\varepsilon,1}(z+\frac{x_{\varepsilon}}{\varepsilon})=\varphi(\varepsilon z)u_{\varepsilon}(z+\frac{x_{\varepsilon}}{\varepsilon})$, then $\{\widehat{u}_{\varepsilon}\}$ is bounded in $H^s(\mathbb{R}^3)$ by virtue of $(V_0)$. Thus, up to a subsequence, we may assume that there exists a $\widehat{u}\in H^s(\mathbb{R}^3)$ such that $\widehat{u}_{\varepsilon}\rightharpoonup\widehat{u}$ in $H^s(\mathbb{R}^3)$, $\widehat{u}_{\varepsilon}\rightarrow\widehat{u}$ in $L_{loc}^p(\mathbb{R}^3)$ for $1\leq p<2_s^{\ast}$, $\widehat{u}_{\varepsilon}\rightarrow\widehat{u}$ a.e. in $\mathbb{R}^3$ and $\widehat{u}$ satisfies
\begin{equation}\label{equ4-13-1}
(-\Delta)^s v+V(x_0) v+\phi_{v}^t v=g(v)   \quad z\in \mathbb{R}^3.
\end{equation}

We now claim that
\begin{equation}\label{equ4-14}
\lim_{\varepsilon\rightarrow0}\sup_{y\in\mathbb{R}^3}\int_{B_1(y)}|\widehat{u}_{\varepsilon}-\widehat{u}|^2\,{\rm d}z=0.
\end{equation}
Suppose the contrary that there exists $\widehat{y}_{\varepsilon}\in\mathbb{R}^3$ such that
\begin{equation}\label{equ4-15}
\lim_{\varepsilon\rightarrow0}\int_{B_1(\widehat{y}_{\varepsilon})}|\widehat{u}_{\varepsilon}-\widehat{u}|^2\,{\rm d}z>0.
\end{equation}
Since $\widehat{u}_{\varepsilon}\rightarrow\widehat{u}$ in $L_{loc}^p(\mathbb{R}^3)$ for $1\leq p<2_s^{\ast}$, we have $\{\widehat{y}_{\varepsilon}\}\subset\mathbb{R}^3$ must be unbounded. Thus, up to a subsequence, still denoted by $\{\widehat{y}_{\varepsilon}\}$, we may assume that $|\widehat{y}_{\varepsilon}|\rightarrow+\infty$ as $\varepsilon\rightarrow0$. Therefore,
\begin{equation}\label{equ4-22}
\lim_{\varepsilon\rightarrow0}\int_{B_1(\widehat{y}_{\varepsilon})}|\widehat{u}|^2\,{\rm d}z=0,\quad \lim_{\varepsilon\rightarrow0}\int_{B_1(\widehat{y}_{\varepsilon})}|\widehat{u}_{\varepsilon}|^2\,{\rm d}z>0.
\end{equation}

Since $\varphi(z)=0$ for $|z|\geq2\beta$, so $|\widehat{y}_{\varepsilon}|\leq\frac{3\beta}{\varepsilon}$ for $\varepsilon$ small. If $|\widehat{y}_{\varepsilon}|\geq\frac{\beta}{2\varepsilon}$, then $\widehat{y}_{\varepsilon}\in B_{3\beta/\varepsilon}(0)\backslash B_{\beta/2\varepsilon}(0)$, and by \eqref{equ4-3}, we get
\begin{align*}
\liminf_{\varepsilon\rightarrow0}\int_{B_1(\widehat{y}_{\varepsilon})}|\widehat{u}_{\varepsilon}|^2\,{\rm d}z&\leq\liminf_{\varepsilon\rightarrow0}\sup_{y\in B_{3\beta/\varepsilon}(0)\backslash B_{\beta/2\varepsilon}(0)}\int_{B_1(y)}|u_{\varepsilon}(z+\frac{x_{\varepsilon}}{\varepsilon})|^2\,{\rm d}z\\
&\leq\liminf_{\varepsilon\rightarrow0}\sup_{y\in A_{\varepsilon}}\int_{B_1(y)}|\widehat{u}_{\varepsilon}|^2\,{\rm d}z=0
\end{align*}
which contradicts with \eqref{equ4-22}. Thus $|\widehat{y}_{\varepsilon}|\leq\frac{\beta}{2\varepsilon}$ for $\varepsilon>0$ small. Without loss of generality, we may assume that $\varepsilon\widehat{y}_{\varepsilon}\rightarrow z_0\in \overline{B_{\beta/2}(0)}$ and $\widetilde{u}_{\varepsilon}\rightharpoonup\widetilde{u}$ in $H^s(\mathbb{R}^3)$, where $\widetilde{u}_{\varepsilon}(z):=\widehat{u}_{\varepsilon}(z+\widehat{y}_{\varepsilon})$. Obviously, $\widetilde{u}\neq0$. It is easy to check that $\widetilde{u}$ satisfies that
\begin{equation*}
(-\Delta)^s v+V(x_0+z_0) v+\phi_{v}^tv=g(v)\quad \text{in}\,\,\mathbb{R}^3.
\end{equation*}
Similarly as in the proof of the case $v\neq0$ of the claim \eqref{equ4-3}, we can get a contradiction for $d_0$ sufficient small. Hence, the claim \eqref{equ4-14} holds and so using the Vanishing Lemma \ref{lem2-3}, we see that
\begin{equation}\label{equ4-23}
\widehat{u}_{\varepsilon}\rightarrow \widehat{u}\quad \text{in}\,\, L^p(\mathbb{R}^3), \,\, p\in(2,2_s^{\ast}).
\end{equation}

By \eqref{equ4-13}, recalling that $\widehat{u}_{\varepsilon}(z)=u_{\varepsilon,1}(z+\frac{x_{\varepsilon}}{\varepsilon})$, we have
\begin{align*}
P_{\varepsilon}(\widehat{u}_{\varepsilon})\leq c_{V_0}+o(1).
\end{align*}
Letting $\varepsilon\rightarrow0$, and using \eqref{equ4-23}, $(V_0)$, we get
\begin{align*}
\mathcal{I}_{V(x_0)}(\widehat{u})\leq c_{V_0}.
\end{align*}
On the other hand, in view of $\langle\mathcal{J}'_{\varepsilon}(u_{\varepsilon}),u_{\varepsilon,1}\rangle\rightarrow0$ and \eqref{equ4-13-0}, and $\langle Q_{\varepsilon}'(u_{\varepsilon}),u_{\varepsilon,1}\rangle=0$, we deduce that
\begin{align*}
\int_{\mathbb{R}^3}|D_s\widehat{u}_{\varepsilon}|^2\,{\rm d}z&+\int_{\mathbb{R}^3}V(\varepsilon z+x_{\varepsilon})|\widehat{u}_{\varepsilon}|^2\,{\rm d}z+\int_{\mathbb{R}^3}\phi_{\widehat{u}_{\varepsilon}}^t|\widehat{u}_{\varepsilon}|^2\,{\rm d}z=\int_{\mathbb{R}^3}f(\varepsilon z,\widehat{u}_{\varepsilon})\widehat{u}_{\varepsilon}\,{\rm d}z+o(1),
\end{align*}
then by Fatou's Lemma, \eqref{equ4-23} and \eqref{equ4-13-1}, we have that
\begin{align*}
&\int_{\mathbb{R}^3}|D_s\widehat{u}|^2\,{\rm d}z+\int_{\mathbb{R}^3}V(x_0)|\widehat{u}|^2\,{\rm d}z+\int_{\mathbb{R}^3}\phi_{\widehat{u}}^t|\widehat{u}|^2\,{\rm d}z\\
&\leq\liminf_{\varepsilon\rightarrow0}\Big(\int_{\mathbb{R}^3}|D_s\widehat{u}_{\varepsilon}|^2\,{\rm d}z+\int_{\mathbb{R}^3}V(\varepsilon z+x_{\varepsilon})|\widehat{u}_{\varepsilon}|^2\,{\rm d}z+\int_{\mathbb{R}^3}\phi_{\widehat{u}_{\varepsilon}}^t|\widehat{u}_{\varepsilon}|^2\,{\rm d}z\Big)\\
&=\liminf_{\varepsilon\rightarrow0}\int_{\mathbb{R}^3}f(\varepsilon z,\widehat{u}_{\varepsilon})\widehat{u}_{\varepsilon}\,{\rm d}z=\int_{\mathbb{R}^3}g(\widehat{u})\widehat{u}\,{\rm d}z\\
&=\int_{\mathbb{R}^3}|D_s\widehat{u}|^2\,{\rm d}z+\int_{\mathbb{R}^3}V(x_0)|\widehat{u}|^2\,{\rm d}z+\int_{\mathbb{R}^3}\phi_{\widehat{u}}^t|\widehat{u}|^2\,{\rm d}z,
\end{align*}
which implies that
\begin{equation*}
\int_{\mathbb{R}^3}|D_s\widehat{u}_{\varepsilon}|^2\,{\rm d}z\rightarrow\int_{\mathbb{R}^3}|D_s\widehat{u}|^2\,{\rm d}z,
\end{equation*}
and
\begin{equation*}
\int_{\mathbb{R}^3}V(\varepsilon z+x_{\varepsilon})|\widehat{u}_{\varepsilon}|^2\,{\rm d}z\rightarrow\int_{\mathbb{R}^3}V(x_0)|\widehat{u}|^2\,{\rm d}z.
\end{equation*}
Hence, by $(V_0)$, we can deduce that
\begin{equation}\label{equ4-24}
\widehat{u}_{\varepsilon}\rightarrow\widehat{u}\quad  \text{in}\,\, H^s(\mathbb{R}^3).
\end{equation}
By \eqref{equ4-2}, \eqref{equ4-23}, it is easy to check that $\widehat{u}\neq0$. By \eqref{equ4-13-1}, we have $\mathcal{I}_{V(x_0)}(\widehat{u})\geq c_{V(x_0)}$. Hence, $\mathcal{I}_{V(x_0)}(\widehat{u})= c_{V(x_0)}$ is proved. In view of $x_0\in\mathcal{M}^{\beta}\subset\Lambda$, we have that $V(x_0)=V_0$ and $x_0\in\mathcal{M}$. As a consequence, $\widehat{u}$ is, up to a translation in the $x-$variable, an element of $\mathcal{L}_{V_0}$, namely there exists $W\in\mathcal{L}_{V_0}$ and $z_0\in\mathbb{R}^3$ such that $\widehat{u}(z)=W(z-z_0)$. Consequently, from \eqref{equ4-2}, \eqref{equ4-13-0} and \eqref{equ4-24}, we have that
\begin{equation*}
\|u_{\varepsilon}-\varphi_{\varepsilon}(\cdot-\frac{x_{\varepsilon}}{\varepsilon}-z_0)W(\cdot-\frac{x_{\varepsilon}}{\varepsilon}-z_0)\|_{H_{\varepsilon}}\rightarrow0\quad \text{as}\,\, \varepsilon\rightarrow0.
\end{equation*}
Observing that $\varepsilon(\frac{x_{\varepsilon}}{\varepsilon}+z_0)\rightarrow x_0\in\mathcal{M}$ as $\varepsilon\rightarrow0$, so the proof is completed.
\end{proof}

For $a\in\mathbb{R}$ we define the sublevel set of $\mathcal{J}_{\varepsilon}$ as follows
\begin{equation*}
\mathcal{J}_{\varepsilon}^a=\{u\in H_{\varepsilon}\,\,\Big|\,\, \mathcal{J}_{\varepsilon}(u)\leq a\}.
\end{equation*}

We observe that the result of Lemma \ref{lem4-1} holds for $d_0>0$ sufficiently small independently of the sequences satisfying the assumptions.
\begin{lemma}\label{lem4-2}
Let $d_0$ be the number given in Lemma \ref{lem4-1}. Then for any $d\in(0,d_0)$, there exist positive constants $\varepsilon_d>0$, $\rho_d>0$ and $\alpha_d>0$ such that
\begin{equation*}
\|\mathcal{J}'_{\varepsilon}(u)\|_{(H_{\varepsilon})'}\geq\alpha_d>0\quad \text{for every}\,\, u\in\mathcal{J}_{\varepsilon}^{c_{V_0}+\rho_d}\cap(\mathcal{N}_{\varepsilon}^{d_0}\backslash\mathcal{N}_{\varepsilon}^d)\,\,\text{and}\,\,\varepsilon\in(0,\varepsilon_d).
\end{equation*}
\end{lemma}
\begin{proof}
By contradiction we suppose that for some $d\in(0,d_0)$, there exists $\{\varepsilon_i\}$, $\{\rho_i\}$ and $u_i\in \mathcal{J}_{\varepsilon_i}^{c_{V_0}+\rho_i}\cap (\mathcal{N}_{\varepsilon_i}^{d_0}\backslash\mathcal{N}_{\varepsilon_i}^d)$ such that
\begin{equation*}
\|\mathcal{J}_{\varepsilon_i}'(u_i)\|_{(H_{\varepsilon_i})'}\rightarrow0\quad \text{as}\,\, i\rightarrow\infty.
\end{equation*}
By Lemma \ref{lem4-1}, we can find $\{y_i\}\subset\mathbb{R}^3$, $x_0\in\mathcal{M}$, $W\in \mathcal{L}_{V_0}$ such that
\begin{equation*}
\lim_{i\rightarrow\infty}|\varepsilon_iy_i-x_0|\rightarrow0\quad \lim_{i\rightarrow\infty}\|u_i-\varphi_{\varepsilon_i}(\cdot-y_i)W(\cdot-y_i)\|_{H_{\varepsilon_i}}=0.
\end{equation*}
Thus, $\varepsilon_iy_i\in\mathcal{M}^{\beta}$ for sufficiently large $i$ and then by the definition of $\mathcal{N}_{\varepsilon_i}$ and $\mathcal{N}_{\varepsilon_i}^d$, we obtain that $\varphi_{\varepsilon_i}(\cdot-y_i)W(\cdot-y_i)\in\mathcal{N}_{\varepsilon_i}$ and $u_i\in \mathcal{N}_{\varepsilon_i}^d$ for sufficiently large $i$. This contradicts with $u_i\not\in\mathcal{N}_{\varepsilon_i}^d$ and completes the proof.
\end{proof}
We recall the definition \eqref{equ4-1} of $\gamma_{\varepsilon}(\tau)$. The following Lemma holds.

\begin{lemma}\label{lem4-3}
There exists $M_0>0$ such that for any $\delta>0$ small, there exists $\alpha_{\delta}>0$ and $\varepsilon_{\delta}>0$ such that if $\mathcal{J}_{\varepsilon}(\gamma_{\varepsilon}(\tau))\geq c_{V_0}-\alpha_{\delta}$ and $\varepsilon\in(0,\varepsilon_{\delta})$, then $\gamma_{\varepsilon}(\tau)\in\mathcal{N}_{\varepsilon}^{M_0\delta}$.
\end{lemma}
\begin{proof}
First, for any $u\in H^s(\mathbb{R}^3)$, we have that
\begin{align*}
\int_{\mathbb{R}^3}|D_s(\varphi_{\varepsilon}u)|^2\,{\rm d}z&\leq2\int_{\mathbb{R}^3}\varphi_{\varepsilon}^2|D_su|^2\,{\rm d}z+2\int_{\mathbb{R}^3}u^2|D_s\varphi_{\varepsilon}|^2\,{\rm d}z\\
&\leq2\int_{\mathbb{R}^3}|D_su|^2\,{\rm d}z+\Big(\int_{\mathbb{R}^3}|u|^{2_s^{\ast}}\,{\rm d}z\Big)^{\frac{2}{2_s^{\ast}}}\Big(\int_{\mathbb{R}^3}|D_s\varphi_{\varepsilon}|^{\frac{3}{s}}\Big)^{\frac{2s}{3}}\\
&\leq 2\int_{\mathbb{R}^3}|D_su|^2\,{\rm d}z+C\Big(\int_{\mathbb{R}^3}|u|^{2_s^{\ast}}\,{\rm d}z\Big)^{\frac{2}{2_s^{\ast}}}.
\end{align*}
Thus, there exists $M_0>0$ such that
\begin{equation}\label{equ4-25}
\|\varphi_{\varepsilon}u\|_{H_{\varepsilon}}\leq M_0\|u\|.
\end{equation}
The remain proof is similar to the proof of Lemma 4.5 in \cite{HL1}, we omit its proof.
\end{proof}
We are now ready to show that the penalized functional $\mathcal{J}_{\varepsilon}$ possesses a critical point for every $\varepsilon>0$ sufficiently small. Choose $\delta_1>0$ such that $M_0\delta_1<\frac{d_0}{4}$ in Lemma \ref{lem4-3}, and fixing $d=\frac{d_0}{4}:=d_1$ in Lemma \ref{lem4-2}. Similar to the proof of Lemma 4.6 in \cite{HL1}, we can prove the following result.
\begin{lemma}\label{lem4-4}
There exists $\overline{\varepsilon}>0$ such that for each $\varepsilon\in(0,\overline{\varepsilon})$, there exists a sequence $\{u_{\varepsilon,n}\}\subset \mathcal{J}_{\varepsilon}^{\widetilde{\mathcal{C}}_{\varepsilon}+\varepsilon}\cap\mathcal{N}_{\varepsilon}^{d_0}$ such that $\mathcal{J}'_{\varepsilon}(u_{\varepsilon,n})\rightarrow0$ in $(H_{\varepsilon})'$ as $n\rightarrow\infty$.
\end{lemma}
\begin{lemma}\label{lem4-5}
$\mathcal{J}_{\varepsilon}$ possesses a nontrivial critical point $u_{\varepsilon}\in\mathcal{N}_{\varepsilon}^{d_0}\cap\mathcal{J}_{\varepsilon}^{\mathcal{D}_{\varepsilon}+\varepsilon}$ for $\varepsilon\in(0,\bar{\varepsilon}]$.
\end{lemma}
\begin{proof}
By Lemma \ref{lem4-4}, there exists $\bar{\varepsilon}>0$ such that for each $\varepsilon\in(0,\bar{\varepsilon}]$, there exists a sequence $\{u_{\varepsilon,n}\}\subset \mathcal{J}_{\varepsilon}^{\mathcal{D}_{\varepsilon}+\varepsilon}\cap\mathcal{N}_{\varepsilon}^{d_0}$ such that $\mathcal{J}_{\varepsilon_n}'(u_{\varepsilon,n})\rightarrow0$ as $n\rightarrow\infty$ in $(H_{\varepsilon})'$. Since $\mathcal{N}_{\varepsilon}^{d_0}$ is bounded, then $\{u_{\varepsilon,n}\}$ is bounded in $H_{\varepsilon}$ and up to a subsequence, we may assume that there exists $u_{\varepsilon}\in H_{\varepsilon}$ such that $u_{\varepsilon,n}\rightharpoonup u_{\varepsilon}$ in $H_{\varepsilon}$, $u_{\varepsilon,n}\rightarrow u_{\varepsilon}$ in $L_{loc}^p(\mathbb{R}^3)$ for $1\leq p<2_s^{\ast}$ and $u_{\varepsilon,n}\rightarrow u_{\varepsilon}$ a.e. in $\mathbb{R}^3$.

We claim that
\begin{equation}\label{equ4-26}
\lim_{R\rightarrow\infty}\sup_{n\geq1}\int_{|x|\geq R}(|D_su_{\varepsilon,n}|^2+V(\varepsilon z)|u_{\varepsilon,n}|^2)\,{\rm d}z=0.
\end{equation}

Indeed, Choosing a cutoff function $\psi_{\rho}\in C^{\infty}(\mathbb{R}^3)$ such that $\psi_{\rho}(z)=1$ on $\mathbb{R}^3\backslash B_{2\rho}(0)$, $\psi_{\rho}(z)=0$ on $B_{\rho}(0)$, $0\leq\psi_{\rho}\leq1$ and $|\nabla \psi_{\rho}|\leq\frac{C}{\rho}$. Since $\psi_{\rho}u_{\varepsilon,n}\in H_{\varepsilon}$, then $\langle\mathcal{J}_{\varepsilon_n}'(u_{\varepsilon,n}),\psi_{\rho}u_{\varepsilon,n}\rangle\rightarrow0$ as $n\rightarrow\infty$. Thus, for sufficiently large $\rho$ such that $\Lambda_{\varepsilon}\subset B_{\rho}(0)$, we have
\begin{align*}
&\int_{\mathbb{R}^3}(|D_su_{\varepsilon,n}|^2+V(\varepsilon z)|u_{\varepsilon,n}|^2)\psi_{\rho}\,{\rm d}z+\int_{\mathbb{R}^3}\int_{\mathbb{R}^3}\frac{(u_{\varepsilon,n}(z)-u_{\varepsilon,n}(y))(\psi_{\rho}(z)-\psi_{\rho}(y))u_{\varepsilon, n}(y)}{|z-y|^{3+2s}}\,{\rm d}y\,{\rm d}z\\
&=\int_{\mathbb{R}^3}f(\varepsilon z,u_{\varepsilon,n})u_{\varepsilon,n}\psi_{\rho}\,{\rm d}z-\int_{\mathbb{R}^3}\phi_{u_{\varepsilon,n}}^t|u_{\varepsilon,n}|^2\psi_{\rho}\,{\rm d}z-4\Big(\int_{\mathbb{R}^3\backslash\Lambda_{\varepsilon}}|u_{\varepsilon,n}|^2\,{\rm d}z-\varepsilon\Big)_{+}\int_{\mathbb{R}^3\backslash\Lambda_{\varepsilon}}|u_{\varepsilon,n}|^2\psi_{\rho}\,{\rm d}z\\
&\leq\int_{\mathbb{R}^3}f(\varepsilon z,u_{\varepsilon,n})u_{\varepsilon,n}\psi_{\rho}\,{\rm d}z\leq\frac{V_0}{k}\int_{\mathbb{R}^3}|u_{\varepsilon,n}|^2\psi_{\rho}\,{\rm d}z.
\end{align*}
In view of the fact that $|D_s\psi_{\rho}|^2\leq\frac{C}{\rho^{2s}}$ for any $z\in\mathbb{R}^3$ and H\"{o}lder's inequality, we deduce that
\begin{align*}
&\int_{\mathbb{R}^3}\int_{\mathbb{R}^3}\frac{(u_{\varepsilon,n}(z)-u_{\varepsilon,n}(y))(\psi_{\rho}(z)-\psi_{\rho}(y))u_{\varepsilon, n}(y)}{|z-y|^{3+2s}}\,{\rm d}y\,{\rm d}z\\
&\leq\Big(\int_{\mathbb{R}^3}|D_su_{\varepsilon,n}|^2\,{\rm d}z\Big)^{\frac{1}{2}}\Big(\int_{\mathbb{R}^3}|D_s\psi_R|^2|u_{\varepsilon,n}|^2\,{\rm d}z\Big)^{\frac{1}{2}}\leq \frac{C}{\rho^s}\|u_{\varepsilon,n}\|_2\leq\frac{C}{\rho^s}.
\end{align*}
Therefore, from the estimates above, we obtain
\begin{equation*}
\int_{\mathbb{R}^3\backslash B_{2\rho}(0)}(|D_su_{\varepsilon,n}|^2+V(\varepsilon z)|u_{\varepsilon,n}|^2)\,{\rm d}z\leq\frac{C}{\rho^s}.
\end{equation*}
Thus, the claim follows. From \eqref{equ4-26}, we see that $u_{\varepsilon,n}\rightarrow u_{\varepsilon}$ in $L^2(\mathbb{R}^3)$. By use of interpolation inequality, we conclude that $u_{\varepsilon,n}\rightarrow u_{\varepsilon}$ in $L^p(\mathbb{R}^3)$ for $2\leq p<2_s^{\ast}$. It follows from standard arguments that $u_{\varepsilon,n}\rightarrow u_{\varepsilon}$ in $H_{\varepsilon}$. Since $0\not\in\mathcal{N}_{\varepsilon}^{d_0}$, $u_{\varepsilon}\neq0$ and $u_{\varepsilon}\in\mathcal{N}_{\varepsilon}^{d_0}\cap\mathcal{J}_{\varepsilon}^{\mathcal{D}_{\varepsilon}+\varepsilon}$. The proof is completed.
\end{proof}

\section{Proof of Theorem \ref{thm1-1}.}

From Lemma \ref{lem4-5}, we see that there exists $\bar{\varepsilon}>0$ and $d_0>0$ such that for each $\varepsilon\in(0,\bar{\varepsilon}]$, $u_{\varepsilon}\in\mathcal{N}_{\varepsilon}^{d_0}\cap\mathcal{J}_{\varepsilon}^{\mathcal{D}_{\varepsilon}+\varepsilon}$ is a nontrivial solution of problem
\begin{equation}\label{equ5-1}
(-\Delta)^{s}u+V(\varepsilon z)u+\phi_u^tu+4\Big(\int_{\mathbb{R}^3\backslash\Lambda_{\varepsilon}}u^2\,{\rm d}z-\varepsilon\Big)_{+}\chi_{\mathbb{R}^3\backslash\Lambda_{\varepsilon}}u=f(\varepsilon z,u)\quad \text{in}\,\, \mathbb{R}^3,
\end{equation}
where $\chi_{\mathbb{R}^3\backslash\Lambda_{\varepsilon}}$ is the characterization function of the set $\mathbb{R}^3\backslash\Lambda_{\varepsilon}$. Taking $-u_{\varepsilon}^{-}$ as a test function in \eqref{equ5-1}, we can deduce that $u_{\varepsilon}\geq0$.
Since $u_{\varepsilon}\in\mathcal{N}_{\varepsilon}^{d_0}\cap\mathcal{J}_{\varepsilon}^{\mathcal{D}_{\varepsilon}+\varepsilon}$, by Lemma \ref{lem4-0-1}, we get that $\{u_{\varepsilon}\}$ is uniformly bounded in $\varepsilon\in(0,\bar{\varepsilon}]$ and $\{\mathcal{J}_{\varepsilon}(u_{\varepsilon})\}$ is uniformly bounded from above for all $\varepsilon>0$ small. Thus, it is easy to check that $\{Q_{\varepsilon}(u_{\varepsilon})\}$ uniformly bounded for all $\varepsilon>0$ small.

{\bf Step 1.} We claim that there exists $C>0$ such that for any $\varepsilon\in(0,\bar{\varepsilon}]$
\begin{equation}\label{equ5-2}
\|u_{\varepsilon}\|_{L^{\infty}(\mathbb{R}^3)}\leq C.
\end{equation}
It suffices to prove that for any $\{\varepsilon_i\}$ satisfying $\varepsilon_i\rightarrow0$, there holds $\|u_{\varepsilon_i}\|_{L^{\infty}(\mathbb{R}^3)}\leq C$, where $C>0$ is a positive constant. For a sequence $\varepsilon_i\rightarrow0$, there is a corresponding sequence $\{u_{\varepsilon_i}\}$ satisfying $u_{\varepsilon_i}\in\mathcal{N}_{\varepsilon_i}^{d_0}\cap\mathcal{J}_{\varepsilon_i}^{\mathcal{D}_{\varepsilon_i}+\varepsilon_i}$ and $\mathcal{J}_{\varepsilon_i}'(u_{\varepsilon_i})=0$. In view of Lemma \ref{lem4-0-3}, we see that $\{u_{\varepsilon_i}\}$ satisfies the condition of Lemma \ref{lem4-1}. Hence
there exist $x_0\in\mathcal{M}$ and $W_0\in\mathcal{L}_{V_0}$ satisfying
\begin{equation*}
\varepsilon_ix_{\varepsilon_i}\rightarrow x_0\quad \text{and}\quad \|u_{\varepsilon_i}-\varphi_{\varepsilon_i}(\cdot-x_{\varepsilon_i})W_0(\cdot-x_{\varepsilon_i})\|_{H_{\varepsilon_i}}\rightarrow0
\end{equation*}
as $i\rightarrow\infty$. Thus
\begin{equation*}
\lim_{i\rightarrow\infty}\|u_{\varepsilon_i}(\cdot+x_{\varepsilon_i})-W_0\|\leq\lim_{i\rightarrow\infty}\|u_{\varepsilon_i}-\varphi_{\varepsilon_i}(\cdot-x_{\varepsilon_i})W_0(\cdot-x_{\varepsilon_i})\|_{H_{\varepsilon_i}}+\lim_{i\rightarrow\infty}\|(1-\varphi_{\varepsilon_i})W_0\|=0.
\end{equation*}
By Lemma \ref{lem2-1-0}, we conclude that $\|u_{\varepsilon_i}(\cdot+x_{\varepsilon_i})\|_{L^{\infty}(\mathbb{R}^3)}\leq C$ and the claim holds true.

{\bf Step 2.} For any sequence $\{\varepsilon_i\}$ with $\varepsilon_i\rightarrow0$, by Lemma \ref{lem4-1}, there exist, up to a subsequence, $\{x_{\varepsilon_i}\}\subset\mathbb{R}^3$, $x_0\in\mathcal{M}$, $W_0\in\mathcal{L}_{V_0}$ such that
\begin{equation*}
\varepsilon_ix_{\varepsilon_i}\rightarrow x_0\quad \text{and}\quad \|u_{\varepsilon_i}-\varphi_{\varepsilon_i}(\cdot-x_{\varepsilon_i})W_0(\cdot-x_{\varepsilon_i})\|_{H_{\varepsilon_i}}\rightarrow0
\end{equation*}
which implies that
\begin{equation}\label{equ5-3}
w_{\varepsilon_i}(z):=u_{\varepsilon_i}(z+x_{\varepsilon_i})\rightarrow W_0\quad \text{in}\,\, H^s(\mathbb{R}^3).
\end{equation}
By \eqref{equ5-2}, we see that $w_{\varepsilon_i}\rightarrow W_0$ in $L^p(\mathbb{R}^3)$ for $1\leq p<\infty$.

Now, setting
\begin{align*}
h_{\varepsilon_i}(z)=w_{\varepsilon_i}(z)&+f(\varepsilon_i z+\varepsilon_i x_{\varepsilon_i},w_{\varepsilon_i}(z))-\Big[V(\varepsilon_i z+\varepsilon_i x_{\varepsilon_i})w_{\varepsilon_i}(z)+\phi_{w_{\varepsilon_i}}^t(z)w_{\varepsilon_i}(z)\\
&+4\Big(\int_{\mathbb{R}^3\backslash\Lambda_{\varepsilon_i}-x_{\varepsilon_i}}w_{\varepsilon_i}^2\,{\rm d}z-\varepsilon_i\Big)_{+}\chi_{\mathbb{R}^3\backslash\Lambda_{\varepsilon_i}-x_{\varepsilon_i}}(z)w_{\varepsilon_i}(z)\Big].
\end{align*}
Clearly, in view of the uniformly boundedness of $Q_{\varepsilon_i}(w_{\varepsilon_i})$ and \eqref{equ5-2}, thus there exists $C>0$ such that $|h_{\varepsilon_i}(z)|\leq C$ for any $z\in\mathbb{R}^3$ and $i\in\mathbb{N}$. By \eqref{equ5-2} and \eqref{equ5-3}, we have that
\begin{align*}
\int_{\mathbb{R}^3}\chi_{\mathbb{R}^3\backslash\Lambda_{\varepsilon_i}-x_{\varepsilon_i}}(z)w_{\varepsilon_i}(z)\,{\rm d}z&\leq\int_{\mathbb{R}^3}|w_{\varepsilon_i}-W_0|\,{\rm d}z+\int_{\mathbb{R}^3}\chi_{\mathbb{R}^3\backslash\Lambda_{\varepsilon_i}-x_{\varepsilon_i}}(z)W_0(z)\,{\rm d}z\\
&=\int_{\mathbb{R}^3\backslash(\Lambda_{\varepsilon_i}-x_{\varepsilon_i})}W_0(z)\,{\rm d}z+o(1)\\
&\leq\int_{\mathbb{R}^3\backslash B_{\beta/\varepsilon_i}(0)}W_0(z)\,{\rm d}z+o(1)\\
&\rightarrow0\quad \text{as}\,\, i\rightarrow\infty.
\end{align*}
Therefore, $h_{\varepsilon_i}\rightarrow h$ in $L^q(\mathbb{R}^3)$ for $1\leq q<\infty$, where $h(z)=W_0(z)+g(W_0)-V(x_0)W_0-\phi_{W_0}^tW_0$.
We rewrite the equation \eqref{equ5-1} as
\begin{equation*}
(-\Delta)^sw_{\varepsilon_i}+w_{\varepsilon_i}=h_{\varepsilon_i}\quad z\in\mathbb{R}^3.
\end{equation*}
According to the arguments in \cite{FQT}, we see that
\begin{equation*}
w_{\varepsilon_i}(z)=\int_{\mathbb{R}^3}\mathcal{K}(z-y)h_{\varepsilon_i}(y)\,{\rm d}y\quad z\in\mathbb{R}^3,
\end{equation*}
where $\mathcal{K}$ is a Bessel potential, which possesses the following properties:\\
$(\mathcal{K}_1)$ $\mathcal{K}$ is positive, radially symmetric and smooth in $\mathbb{R}^3\backslash\{0\}$;\\
$(\mathcal{K}_2)$ there exists a constant $C>0$ such that $\mathcal{K}(x)\leq \frac{C}{|x|^{3+2s}}$ for all $x\in\mathbb{R}^3\backslash\{0\}$;\\
$(\mathcal{K}_3)$ $\mathcal{K}\in L^\tau(\mathbb{R}^3)$ for $\tau\in[1,\frac{3}{3-2s})$.

We define two sets $A_{\delta}=\{y\in\mathbb{R}^3\,\,|\,\, |z-y|\geq\frac{1}{\delta}\}$ and $B_{\delta}=\{y\in\mathbb{R}^3\,\,|\,\, |z-y|<\frac{1}{\delta}\}$.
Hence,
\begin{equation*}
0\leq w_{\varepsilon_i}(z)\leq \int_{\mathbb{R}^3}\mathcal{K}(z-y)|h_{\varepsilon_i}(y)|\,{\rm d}y=\int_{A_{\delta}}\mathcal{K}(z-y)|h_{\varepsilon_i}(y)|\,{\rm d}y+\int_{B_{\delta}}\mathcal{K}(z-y)|h_{\varepsilon_i}(y)|\,{\rm d}y.
\end{equation*}
From the definition of $A_{\delta}$ and $(\mathcal{K}_2)$, we have that for all $n\in\mathbb{N}$,
\begin{equation*}
\int_{A_{\delta}}\mathcal{K}(z-y)|h_{\varepsilon_i}(y)|\,{\rm d}y\leq C\delta^s\|h_{\varepsilon_i}\|_{\infty}\int_{A_{\delta}}\frac{1}{|z-y|^{3+s}}\,{\rm d}y\leq C\delta^s\int_{A_{\delta}}\frac{1}{|z-y|^{3+s}}\,{\rm d}y:=C\delta^{2s}.
\end{equation*}

On the other hand, by H\"{o}lder's inequality and $(\mathcal{K}_3)$, we deduce that
\begin{align*}
&\int_{B_{\delta}}\mathcal{K}(z-y)|h_{\varepsilon_i}(y)|\,{\rm d}y\leq\int_{B_{\delta}}\mathcal{K}(z-y)|h_{\varepsilon_i}-h|\,{\rm d}y+\int_{B_{\delta}}\mathcal{K}(z-y)|h|\,{\rm d}y\\
&\leq\Big(\int_{B_{\delta}}\mathcal{K}^{\frac{6}{3+2s}}\,{\rm d}y\Big)^{\frac{3+2s}{6}}\Big(\int_{B_{\delta}}|h_{\varepsilon_i}-h|^{\frac{6}{3-2s}}\,{\rm d}y\Big)^{\frac{3-2s}{6}}+
\Big(\int_{B_{\delta}}\mathcal{K}^{\frac{6}{3+2s}}\,{\rm d}y\Big)^{\frac{3+2s}{6}}\Big(\int_{B_{\delta}}|h|^{\frac{6}{3-2s}}\,{\rm d}y\Big)^{\frac{3-2s}{6}}\\
&\leq\Big(\int_{\mathbb{R}^3}\mathcal{K}^{\frac{6}{3+2s}}\,{\rm d}y\Big)^{\frac{3+2s}{6}}\Big(\int_{\mathbb{R}^3}|h_{\varepsilon_i}-h|^{\frac{6}{3-2s}}\,{\rm d}y\Big)^{\frac{3-2s}{6}}+
\Big(\int_{\mathbb{R}^3}\mathcal{K}^{\frac{6}{3+2s}}\,{\rm d}y\Big)^{\frac{3+2s}{6}}\Big(\int_{B_{\delta}}|h|^{\frac{6}{3-2s}}\,{\rm d}y\Big)^{\frac{3-2s}{6}},
\end{align*}
where we have used the fact that $\frac{6}{3+2s}<\frac{3}{3-2s}$.

Since $\Big(\int_{B_{\delta}}|h|^{\frac{6}{3-2s}}\,{\rm d}y\Big)^{\frac{3-2s}{6}}\rightarrow0$ as $|z|\rightarrow+\infty$, thus, we deduce that there exist $i_0\in\mathbb{N}$ and $R_0>0$ independence of $\delta>0$ such that
\begin{equation*}
\int_{B_{\delta}}\mathcal{K}(z-y)|h_{\varepsilon_i}(y)|\,{\rm d}y\leq \delta,\quad \forall i\geq i_0\quad \text{and}\quad |z|\geq R_0.
\end{equation*}
Hence,
\begin{equation*}
\int_{\mathbb{R}^3}\mathcal{K}(z-y)|h_{\varepsilon_i}(y)|\,{\rm d}y\leq C\delta^{2s}+\delta,\quad \forall i\geq i_0\quad \text{and}\quad |z|\geq R_0.
\end{equation*}
For each $i\in\{1,2,\cdots,i_0-1\}$, there exists $R_i>0$ such that $\Big(\int_{B_{\delta}}|h_{\varepsilon_i}|^{\frac{6}{3-2s}}\,{\rm d}y\Big)^{\frac{3-2s}{6}}< \delta$ as $|z|\geq R_i$. Thus, for $|z|\geq R_i$, we have that
\begin{align*}
\int_{\mathbb{R}^3}\mathcal{K}(z-y)|h_{\varepsilon_i}(y)|\,{\rm d}y&\leq C\delta^{2s}+\int_{B_{\delta}}\mathcal{K}(z-y)|h_{\varepsilon_i}(y)|\,{\rm d}y\\
&\leq C\delta^{2s}+\|\mathcal{K}\|_{\frac{6}{3+2s}}\Big(\int_{B_{\delta}}|h_{\varepsilon_i}|^{\frac{6}{3-2s}}\,{\rm d}y\Big)^{\frac{3-2s}{6}}\leq C(\delta^{2s}+\delta)
\end{align*}
for each $i\in\{1,2,\cdots,i_0-1\}$.
Therefore, taking $R=\max\{R_0,R_1,\cdots, R_{i_0-1}\}$, we infer that for any $i\in\mathbb{N}$, there holds
\begin{equation*}
0\leq w_{\varepsilon_i}(z)\leq\int_{\mathbb{R}^3}\mathcal{K}(z-y)|h_{\varepsilon_i}(y)|\,{\rm d}y\leq C\delta^{2s}+\delta,\quad \text{for all}\quad |z|\geq R
\end{equation*}
which implies that $\lim\limits_{|z|\rightarrow\infty}w_{\varepsilon_i}(z)=0$ uniformly in $i\in\mathbb{N}$.

Similar arguments to the proof of $(ii)$ of Proposition \ref{pro3-4}, we see that for any $i\in\mathbb{N}$ but fixed, there exists $C>0$ independent of $\varepsilon_i>0$ such that
\begin{equation*}
0\leq w_{\varepsilon_i}(z)\leq\frac{C}{1+|z|^{3+2s}} \quad \text{for any}\,\, z\in\mathbb{R}^3.
\end{equation*}
Therefore,
\begin{align*}
\int_{\mathbb{R}^3\backslash\Lambda_{\varepsilon_i}}w_{\varepsilon_i}^2\,{\rm d}z\leq C\int_{\mathbb{R}^3\backslash B_{R/\varepsilon_i}(0)}\frac{1}{(1+|z|^{3+2s})^2}\,{\rm d}z\leq C\varepsilon_i^{4s+3}
\end{align*}
which implies that $Q_{\varepsilon_i}(z)=0$ for $\varepsilon_i>0$ small. Hence $w_{\varepsilon_i}$ is a solution of the following problem
\begin{equation*}
(-\Delta)^{s}u+V(\varepsilon z)u+\phi_u^tu=f(\varepsilon z,u)\quad \text{in}\,\, \mathbb{R}^3.
\end{equation*}

{\bf Step 3.} For the $w_{\varepsilon}$ above in Step 2, $u_{\varepsilon}(z)=w_{\varepsilon}(z-x_{\varepsilon})<a$, for all $z\in\mathbb{R}^3\backslash\Lambda_{\varepsilon}$.

Noting that $\varepsilon x_{\varepsilon}\rightarrow x_0$ and $x_0\in\Lambda$. Thus, there exists $R'>0$ such that $B_{R'}(\varepsilon x_{\varepsilon})\subset\Lambda$ for $\varepsilon>0$ small. Hence, $B_{R'/\varepsilon}(x_{\varepsilon})\subset\Lambda_{\varepsilon}$ for $\varepsilon>0$ small. Moreover, by Step 2, there is $R_1>0$ such that $w_{\varepsilon}(z)<a$ for $|z|\geq R_1$. Thus,
\begin{equation*}
u_{\varepsilon}(z)=w_{\varepsilon}(z-x_{\varepsilon})<a,\quad \text{for all}\,\,z\in \mathbb{R}^3\backslash B_{R_1}(x_{\varepsilon})\quad \text{and}\,\, \varepsilon>0\,\,\text{small}.
\end{equation*}
Since
\begin{equation*}
\mathbb{R}^3\backslash\Lambda_{\varepsilon}\subset\mathbb{R}^3\backslash B_{R'/\varepsilon}(x_{\varepsilon})\subset\mathbb{R}^3\backslash B_{R_1}(x_{\varepsilon})\quad \text{and}\,\, \varepsilon>0\,\,\text{small}
\end{equation*}
and then
\begin{equation*}
u_{\varepsilon}(z)=w_{\varepsilon}(z-x_{\varepsilon})<a\quad \forall z\in\mathbb{R}^3\backslash\Lambda_{\varepsilon}\quad \text{and}\,\, \varepsilon>0\,\,\text{small}.
\end{equation*}

{\bf Step 4.} By Lemma \ref{lem4-5}, we see that problem \eqref{equ5-1} has a nonnegative solution $v_{\varepsilon}$ for all $\varepsilon\in(0,\bar{\varepsilon}]$. From Step 3, there exists $\varepsilon_0>0$ such that
\begin{equation*}
v_{\varepsilon }(z)<a\quad \forall z\in\mathbb{R}^3\backslash\Lambda_{\varepsilon}\,\,\text{and}\,\, \varepsilon\in(0,\varepsilon_0)
\end{equation*}
which implies that $f(\varepsilon z,v_{\varepsilon})=g(v_{\varepsilon})$. Thus, $v_{\varepsilon}$ is a solution of problem
\begin{equation}\label{equ5-4}
(-\Delta)^sv+V(\varepsilon z) v+ \phi_{v}^t v=g(v)\quad z\in\mathbb{R}^3.
\end{equation}
for all $\varepsilon\in(0,\varepsilon_0)$. Let $u_{\varepsilon}(x)=v_{\varepsilon}(x/\varepsilon)$ for every $\varepsilon\in(0,\varepsilon_0)$, it follows that $u_{\varepsilon}$ must be a solution to original problem \eqref{main} for $\varepsilon\in(0,\varepsilon_0)$.

If $y_{\varepsilon}$ denotes a global maximum point of $v_{\varepsilon}$, then
\begin{equation}\label{equ5-5}
v_{\varepsilon}(y_{\varepsilon})\geq a\quad \forall\, \varepsilon\in(0,\varepsilon_0).
\end{equation}
Suppose that $v_{\varepsilon}(y_{\varepsilon})<a$, taking $v_{\varepsilon}$ as a text function for \eqref{equ5-4}, we get
\begin{align*}
V_0\int_{\mathbb{R}^3}v_{\varepsilon}^2\,{\rm d}z&\leq\int_{\mathbb{R}^3}V(\varepsilon z)v_{\varepsilon}^2\,{\rm d}z\leq\int_{\mathbb{R}^3}g(v_{\varepsilon})v_{\varepsilon}\,{\rm d}z\\
&=\int_{\mathbb{R}^3}v_{\varepsilon}^2\frac{g(v_{\varepsilon})}{v_{\varepsilon}^{q-1}}v_{\varepsilon}^{q-2}\,{\rm d}z\\
&\leq\int_{\mathbb{R}^3}v_{\varepsilon}^2\frac{g(a)}{a}\,{\rm d}z=\frac{V_0}{k}\int_{\mathbb{R}^3}v_{\varepsilon}^2\,{\rm d}z
\end{align*}
which we have used the hypothesis $(g_3)$ and $q-2>0$. Hence we get a contradiction owing to the choosing $k>2$. In view of Step 3, we see that $\{y_{\varepsilon}\}$ is bounded for $\varepsilon\in(0,\varepsilon_0)$.

In what follows, setting $z_{\varepsilon}=\varepsilon y_{\varepsilon}+\varepsilon x_{\varepsilon}$, where $\{x_{\varepsilon}\}$ is given in Step 2. Since $u_{\varepsilon}(x)=v_{\varepsilon}(\frac{x}{\varepsilon}-x_{\varepsilon})$, then $z_{\varepsilon}$ is a global maximum point of $u_{\varepsilon}$ and $u_{\varepsilon}(z_{\varepsilon})\geq a$ for all $\varepsilon\in(0,\varepsilon_0)$.

Now, we claim that $\lim\limits_{\varepsilon\rightarrow0^{+}}V(z_{\varepsilon})=V_0$. Indeed, if the above limit does not hold, there is $\varepsilon_n\rightarrow0^{+}$ and $\gamma_0>0$ such that
\begin{equation}\label{equ5-6}
V(z_{\varepsilon_n})\geq V_0+\gamma_0\quad \forall n\in\mathbb{N}.
\end{equation}

By Step 2, we know that $\lim\limits_{|z|\rightarrow\infty}v_{\varepsilon_n}(z)=0$ uniformly in $n\in\mathbb{N}$. From \eqref{equ5-5}, thus $\{z_{\varepsilon_n}\}$ is a bounded sequence. Using Lemma \ref{lem4-1}, we know that there is $x_0\in\mathcal{M}$ such that $V(x_0)=V_0$ and $\varepsilon_nx_{\varepsilon_n}\rightarrow x_0$. Hence, $z_{\varepsilon_n}=\varepsilon_n x_{\varepsilon_n}+\varepsilon_n y_{\varepsilon_n}\rightarrow x_0$ which implies that $V(z_{\varepsilon_n})\rightarrow V(x_0)=V_0$ contradicting with \eqref{equ5-6}.

To complete the proof, we only need to prove the decay properties of $u_{\varepsilon}$. Similar argument to the proof of Proposition \ref{pro3-4}, we can obtain that
\begin{equation*}
0<v_{\varepsilon}(z)\leq \frac{C}{1+|z|^{3+2s}}.
\end{equation*}
Thus, by the boundedness of $\{y_{\varepsilon}\}$, i.e., there exists $C_0>0$ such that $|y_{\varepsilon}|\leq C_0$, we have
\begin{align*}
u_{\varepsilon}(x)=v_{\varepsilon}(\frac{x}{\varepsilon}-x_{\varepsilon})\leq \frac{C}{1+|\frac{x-z_{\varepsilon}+\varepsilon y_{\varepsilon}}{\varepsilon}|^{3+2s}}&\leq\frac{ C\varepsilon^{3+2s}}{\varepsilon^{3+2s}(1-C_0^{3+2s})+|x-z_{\varepsilon}|^{3+2s}}\\
&:=\frac{ C\varepsilon^{3+2s}}{\varepsilon^{3+2s}C_1+|x-z_{\varepsilon}|^{3+2s}}.
\end{align*}

{\bf Acknowledgements.}
The work is supported by NSFC grant 11501403 and fund program for the Scientific Activities of Selected Returned Overseas Professionals in Shanxi Province (2018).

\end{document}